\documentclass[final,onefignum,onetabnum]{siamart171218}

\usepackage{lipsum}
\usepackage{amsfonts}
\usepackage{graphicx}
\usepackage{epstopdf}
\usepackage{algorithmic}
\ifpdf
  \DeclareGraphicsExtensions{.eps,.pdf,.png,.jpg}
\else
  \DeclareGraphicsExtensions{.eps}
\fi

\newsiamremark{remark}{Remark}
\newsiamremark{hypothesis}{Hypothesis}
\crefname{hypothesis}{Hypothesis}{Hypotheses}
\newsiamthm{claim}{Claim}

\headers{Entropic Fourier method for BTE}{Z. Cai, Y. Fan, and L. Ying}

\title{An entropic Fourier method for the Boltzmann equation}
\author{Zhenning Cai\thanks{Department of Mathematics, National University of
    Singapore, Singapore 119076 (\email{matcz@nus.edu.sg}).}
\and Yuwei Fan\thanks{Department of Mathematics, Stanford University, Stanford, CA 94305
    (\email{ywfan@stanford.edu}).}
\and Lexing Ying\thanks{Department of Mathematics and Institute for 
    Computational and Mathematical Engineering, Stanford University, Stanford, CA 94305
    (\email{lexing@stanford.edu}).}
}
\usepackage{amsopn}

\usepackage{srcltx,graphicx}
\usepackage{color}
\usepackage[hang]{subfigure}
\usepackage{overpic}
\usepackage{ulem}
\usepackage{enumitem}
\usepackage{mathrsfs}

\newtheorem{condition}{Condition}

\newcommand\bbN{\mathbb{N}}
\newcommand\bbZ{\mathbb{Z}}
\newcommand\bbP{\mathbb{P}}
\newcommand\bbR{\mathbb{R}}
\newcommand\bbS{\mathbb{S}}

\newcommand\be{{e}}
\newcommand\bg{{g}}

\newcommand\bu{{u}}

\newcommand\bv{{v}}

\newcommand\bx{{x}}

\newcommand\dd{\,\mathrm{d}}

\newcommand\mB{\mathcal{B}}
\newcommand\mD{\mathcal{D}}
\newcommand\mI{\mathcal{I}}

\newcommand\mP{\mathcal{P}}
\newcommand\mQ{\mathcal{Q}}
\newcommand\mS{\mathcal{S}}

\newcommand\bone{\boldsymbol{\mathrm{1}}}

\newcommand\F{F}
\newcommand\f{f}
\newcommand\y{{y}}
\newcommand\z{{z}}

\newcommand\X{{X}}
\newcommand\K{{K}}

\newcommand\T{\mathsf{T}}

\newcommand\G{\mathsf{G}}
\newcommand\C{\mathsf{C}}

\newcommand\fast{\mathsf{fast}}

\newcommand\rmspan{\mathrm{span}}
\newcommand\supp{\mathrm{Supp}}
\newcommand\imag{\boldsymbol{\mathrm{i}}}

\newcommand\pd[2]{\dfrac{\partial {#1}}{\partial {#2}}}
\newcommand\od[2]{\dfrac{\mathrm{d}{#1}}{\mathrm{d}{#2}}}

\graphicspath{{images/}}

\ifpdf
\hypersetup{
  pdftitle={Entropic Fourier method for BTE}
  pdfauthor={Z. Cai, Y. Fan and L. Ying}
}
\fi

\begin{document}

\maketitle

% REQUIRED
\begin{abstract}
    We propose an entropic Fourier method for the numerical discretization of the Boltzmann
    collision operator. The method, which is obtained by modifying a Fourier Galerkin method to
    match the form of the discrete velocity method, can be viewed both as a discrete velocity method
    and as a Fourier method. As a discrete velocity method, it preserves the positivity of the
    solution and satisfies a discrete version of the H-theorem. As a Fourier method, it allows one
    to readily apply the FFT-based fast algorithms. A second-order convergence rate is validated by
    numerical experiments.
\end{abstract}

% REQUIRED

% REQUIRED

\section{Introduction}
Gas kinetic theory describes the statistical behavior of a large number of gas molecules in the
joint spatial and velocity space. It has been widely used to model gases outside the hydrodynamic
regime, for example in the field of rarefied gas dynamics.  Let $\f(t,\bx,\bv)$ be the mass density
distribution of the particles, depending on the time $t\in\bbR^+$, position $\bx\in\bbR^d$ ($d\geq
2$) and microscopic velocity $\bv\in\bbR^d$. Based on the molecular chaos assumption, the Boltzmann
equation 
\begin{equation}\label{eq:Boltzmann}
  \pd{\f}{t}+\bv\cdot\nabla_{\bx}\f=\mQ[\f,\f],\quad
  \f(0,x,v) = \f^0(x,v)
\end{equation}
for the evolution of $\f(t,\bx,\bv)$ was derived in \cite{Boltzmann} and has served as the
fundamental equation in the gas kinetic theory. When modeling the binary interaction between the
particles, the Boltzmann collision operator $\mQ[\f,\f]$ takes the form
\begin{equation}\label{eq:collision}
  \mQ[\f,\f](\bv) = \int_{\bbR^d}\int_{\bbS^{d-1}} \mB(\bv-\bv_*,\omega) \left[
    \f(\bv')\f(\bv_*')-\f(\bv)\f(\bv_*) \right] \dd\bv_*\dd\omega
\end{equation}
for monatomic gases, where
\[
\bv' = \frac{\bv+\bv_*}{2}+\frac{|\bv-\bv_*|}{2}\omega,\qquad
\bv'_* = \frac{\bv+\bv_*}{2}-\frac{|\bv-\bv_*|}{2}\omega
\]
are the post-collisional velocities of two particles with pre-collisional velocities $\bv$ and
$\bv_*$, and $\omega$ is the angular parameter of the collision. Here the variables $t$ and $\bx$
are omitted for simplicity and we shall continue doing so when focusing only on the collision term.
The collision kernel $\mB$ is non-negative and usually takes the form
\begin{equation}\label{eq:Bdef}
\mB(\bv-\bv_*,\omega) = b(|\bv-\bv_*|, \cos\theta),\qquad 
\cos\theta = |(\bv-\bv_*)\cdot\omega|/|\bv-\bv_*|.
\end{equation}

%property %Recalling that $\f(t,\bx,\bv)$ denotes the mass density of the particles,
The Boltzmann equation \eqref{eq:Boltzmann} guarantees that $\f(t,\bx,\bv)$ remains non-negative if
the initial value $\f(t=0,\bx,\bv)$ is also non-negative \cite{Simons1978}. The symmetry of the
collision term \eqref{eq:collision} and the fact $\dd\bv\dd\bv_*=\dd\bv'\dd\bv'_*$ imply that for
any function $\psi(\cdot)$:
\begin{small}
  \begin{equation}
    \begin{aligned}
      &\int_{\bbR^d}\psi(\bv)\mQ[\f,\f](\bv)\dd\bv=\\ 
      &\quad\frac{1}{4}
      \int_{\bbR^d}\int_{\bbR^d}\int_{\bbS^{d-1}}\!  \left(
      \psi(\bv)+\psi(\bv_*)-\psi(\bv')-\psi(\bv'_*) \right) \mB[\f(\bv') \f(\bv'_*)-\f(\bv)
        \f(\bv_*)]\dd\bv_*\dd\bv\dd\omega.
    \end{aligned}
  \end{equation}
\end{small}%
Setting $\psi(\bv)=1,\bv,|\bv|^2$ gives rise to the
conservation of the mass, momentum and energy
\begin{equation}
    \int_{\bbR^d}\mQ[\f,\f](\bv)\dd\bv=0,\quad
    \int_{\bbR^d}\mQ[\f,\f](\bv)\bv\dd\bv=0,\quad
    \int_{\bbR^d}\mQ[\f,\f](\bv)|\bv|^2\dd\bv=0,
\end{equation}
respectively. The famous H-theorem that states the monotonicity of the entropy
\begin{equation}
  \begin{aligned}
    &\int_{\bbR^d}\mQ[\f,\f](\bv)\ln(\f(\bv)) \dd\bv = \\
    &\quad \frac{1}{4}
    \int_{\bbR^d}\int_{\bbR^d}\int_{\bbS^{d-1}}\!
    \ln\left( \frac{\f(\bv)\f(\bv_*)}{\f(\bv')\f(\bv'_*)} \right)
    \mB[\f(\bv') \f(\bv'_*)-\f(\bv)
      \f(\bv_*)]\dd\bv_*\dd\bv\dd\omega
    \leq 0
  \end{aligned}
\end{equation}
can also be obtained by setting $\psi(\bv)=\ln(\f(\bv))$.

%other representation
By introducing new variables $\y=\bv'-\bv$ and $\z=\bv'_*-\bv$ and carrying out algebra
calculations, the Boltzmann collision operator can be rewritten as (see \cite{carleman1957,
  Panferov2002, Mouhot} for details)
\begin{equation} \label{eq:Carleman}
  \mQ[\f,\f](\bv)=\int_{\bbR^d}\int_{\bbR^d}\tilde{\mB}(\y,\z)\delta(\y\cdot \z) [\f(\bv+\y)\f(\bv+\z)
    -\f(\bv)\f(\bv+\y+\z)]\dd \y\dd \z.
\end{equation}
This is the well-known Carleman representation \cite{carleman1957} of the Boltzmann collision
operator, where $\tilde{\mB}(y,z)$ is related to $\mB(\bv-\bv_*,\omega)$ in \eqref{eq:Bdef} by
\begin{equation}
    %\tilde{\mB}(y,z) = 2^{d-1} b\left(\sqrt{|y|^2+|z|^2},\frac{|y|}{\sqrt{|y|^2+|z|^2}}\right)
    %(|y|^2+|z|^2)^{(2-d)/2}.
    \tilde{\mB}(\y,\z) = 2^{d-1} \mB\left(\y+\z,\frac{\y-\z}{|\y-\z|}\right)
    |\y+\z|^{2-d}.
\end{equation}

Though the Boltzmann equation serves as the fundamental equation in gas dynamics, its high
dimensional nature and the complexity of the collision operator pose difficulties for its numerical
solution. A classical method is the direct Monte Carlo simulation \cite{Bird}, which uses simulated
particles to mimic gas molecules and handles the collisions in a stochastic way. Though it treats
the high dimensionality effectively, the convergence order is low and the numerical solution appears
rather oscillatory.

With the rapid growth of computing power, it has become more practical to solve the Boltzmann
equation with deterministic methods. For all deterministic approaches, the complexity of the
collision integral poses the most serious difficulty for numerical computation. Therefore, this
paper focuses on the spatially homogeneous case
\begin{equation}\label{eq:Boltzmann_homogeneous}
  \pd{\f}{t}=\mQ[\f,\f]
\end{equation}
for simplicity.

In the past decades, several deterministic schemes have been developed for the Boltzmann collision
term. Two methods that have attracted the most attention are the discrete velocity method (DVM)
\cite{goldstein1989,rogier1994, Bobylev1995, Panferov2002} and the Fourier Galerkin method (FGM)
\cite{bobylev1996difference, pareschi1996Spectral, pareschi2000Spectral}, which will be reviewed in
what follows.

\subsection{Discrete velocity method} \label{sec:DVM}

The discrete velocity method (DVM) assumes that the particle velocity takes only values from a
finite set. Consider a domain $\mD_T = [-T,T]^d$ for the velocity variable $v$, that is discretized
uniformly with step size $h=2T/N$ for a positive integer $N$ (which is assumed to be odd for
simplicity). By adopting the $d$-dimensional multi-index notation $k=(k_1,\ldots,k_d)$, one can
denote the set of discrete velocity samples by
\begin{equation}\label{eq:X}
  \X=\{h\cdot k | k=(k_1,\ldots,k_d),-n\le k_1,\ldots,k_d\le n\}.
\end{equation}
where $N=2n+1$. In the rest of this paper, we use the lower-case letters $p,q,r$ and $s$ to denote
the discrete velocity samples in $\X$. 
%In order to simplify the notations, when we write $\sum_{p},
%\sum_{q}, \sum_{r}, \sum_{s}$, it is understood implicitly that the sum is taken over the set $X$.

Using $\F_r(t)$ for $r\in\X$ as the numerical approximations of the distribution function
$\f(t,\bv)$ at the points in $\X$, the governing equations of DVM for $\F_r(t)$ are
\begin{equation}\label{eq:dvmcollision}
  \od{\F_r(t)}{t}=Q_r(t) :=\sum_{p,q,s\in\X}A_{pq}^{rs}\left( \F_p(t)\F_q(t)-\F_r(t)\F_s(t) \right),\quad
  r \in X.
\end{equation}
Here $Q_r(t)$ serves as an approximation to $\mQ[\f,\f](t,r)$. The coefficients $A_{pq}^{rs}$ are
non-negative constants and satisfy the conservation relations
\begin{equation}\label{eq:conservationA}
  A_{pq}^{rs} \neq 0
  \quad \Rightarrow \quad
  p+q=r+s\quad\text{and}\quad|p|^2+|q|^2=|r|^2+|s|^2
\end{equation}
and the symmetry property
\begin{equation}\label{eq:symmetryrelationA}
  A_{pq}^{rs}=A_{qp}^{rs}=A_{rs}^{pq}.
\end{equation}

The property \eqref{eq:conservationA} shows that the collisions in DVM also satisfy the momentum and
energy conservation, and the property \eqref{eq:symmetryrelationA} implies that the collisions are
also reversible as in the Boltzmann equation. These two facts guarantee that DVM maintains a number
of fundamental physical properties of the continuous Boltzmann equation, such as (a) the positivity
of the distribution function, (b) the exact conservation of mass, momentum and energy, and (c) a
discrete H-theorem.

More precisely, the values $\F_r(t)$ for $r\in\X$ are always non-negative if the initial values
$\F_r(t=0)$ are non-negative \cite{Simons1978,platkowski1988discrete}. 
The symmetry relation \eqref{eq:symmetryrelationA} implies that
\begin{equation}\label{eq:dvm_sym}
  \sum_{r\in\X} Q_r\psi_r=\frac{1}{4}\sum_{p,q,r,s\in\X} A_{pq}^{rs}\left(
  \psi_r+\psi_s-\psi_p-\psi_q \right) \left( \F_p\F_q-\F_r\F_s \right).
\end{equation}
Combining \eqref{eq:dvm_sym} and the relations \eqref{eq:conservationA} gives rise to the
conservation of mass, momentum and energy in the discrete sense
\begin{equation}
  \sum_{r\in\X}Q_r =0, \quad \sum_{r\in\X}Q_r r =0, \quad \sum_{r\in\X}Q_r |r|^2 = 0.
\end{equation}
After letting $\psi_r = \ln(\F_r)$ in \eqref{eq:dvm_sym}, one obtains a discrete version of the
H-theorem for DVM
\begin{equation}
  \sum_{r\in\X}Q_r\ln(\F_r) =\frac{1}{4}\sum_{p,q,r,s\in\X} A_{pq}^{rs} \ln\left(
  \frac{\F_r\F_s}{\F_p\F_q} \right) \left( \F_p\F_q-\F_r\F_s \right)
  \leq 0
\end{equation}
using the non-negativity of the coefficients $A_{pq}^{rs}$ and the monotonicity of the $\ln$
function. Notice that, in this argument, the symmetry relations \eqref{eq:symmetryrelationA}, the
non-negativity of $A_{pq}^{rs}$, and the non-negativity of the initial values are all essential for
derivation of the H-theorem.

As a direct discretization for the high-dimensional integral of the collision term, DVM has a rather
high computational cost $O(N^{2D+\delta})$ for some $0<\delta\leq 1$
\cite{parischi2014numerical}. It is also difficult to achieve decent convergence rate due to the
insufficiency collision pairs in the Cartesian grid used for velocity
discretization. More precisely, in the 2D case, the rate of convergence of DVM introduced in
\cite{goldstein1989} is only $O((1/\log h)^p)$ with $p<1/2-1/\pi$ \cite{Fainsilber2006}. For the 3D
case, the best rate of convergence of DVM is also slower than the first order \cite{Palczewski1997,
  Panferov2002}.

The method in \cite{morris2008improvement, varghese2007arbitrary} tries to improve the accuracy of
DVM by interpolation. While the mass, momentum and energy are conserved in this scheme, positivity
and the H-theorem fail to hold. The fast algorithm in \cite{MouhotFastDVM} reduces the computational
cost of DVM to $O(\bar{N}^dN^d\log(N))$ with some parameter $\bar{N}\ll N$ for the hard sphere
molecules, but it abandons the conservation of momentum and energy.

\subsection{Fourier Galerkin method} \label{sec:FGM}

The Fourier-based methods assume that the distribution function $\f(t,v)$ is
supported (in the $v$ variable) in a ball $B_{R/2}$ centered at the origin with radius $R / 2$.
Under this assumption, it makes sense to focus on the functions $f(t,v)$ with $\supp(f)\subset B_{R
/ 2}$. For those functions, $\supp(\mQ[\f,\f])\subset B_{\sqrt{2}R/2}$ and the collision term
$\mQ[\f,\f]$ reduces to a truncated version $\mQ^R[\f,\f]$ defined as
\begin{equation}\label{eq:approximated_collision}
  \mQ^R[\f,\f](\bv) := \int_{B_R}\int_{B_R}\tilde{\mB}(\y,\z)\delta(\y\cdot \z)[\f(\bv+\y)
    \f(\bv+\z) - \f(\bv)\f(\bv+\y+\z)]\dd \y\dd \z,
\end{equation}
where the superscript in $\mQ^R[f,f]$ denotes the truncation radius. In order to obtain a spectral
approximation to the collision term, one restricts the domain of the distribution function $\f(v)$
to the cube $\mD_T = [-T,T]^d$ with $T\geq \frac{3\sqrt{2}+1}{4}R$ in order to reduce aliasing. One
then extends it periodically to the whole space. (See \cite{pareschi2000Spectral, Mouhot} for
details of the derivation). After periodization, $\f(\bv)$ can be written as a Fourier series
\begin{equation}\label{eq:fourierseries}
  \f(\bv) = \sum_{k\in\bbZ^d}\hat{\f}_k E_k(\bv),
  \qquad
  \hat{\f}_k =\frac{1}{(2T)^d}\int_{\mD_T}\f(\bv)    E_{-k}(\bv) \dd\bv,
\end{equation}
where $E_k(\bv)=\exp\left( \frac{\imag\pi}{T}k\cdot\bv \right)$.  Substituting
\eqref{eq:fourierseries} into \eqref{eq:approximated_collision} gives rise to the following
representation of the truncated collision operator:
\begin{equation}\label{eq:FGM_Qv}
  \mQ^R[\f,\f](\bv)=\sum_{l,m\in\bbZ^d}\left( \hat{B}(l,m)-\hat{B}(m,m) \right)\hat{\f}_l\hat{\f}_m
  E_{l+m}(\bv), \quad v\in\mD_T,
\end{equation}
where
\begin{equation}\label{eq:def_hatB}
  \hat{B}(l,m) := \int_{B_R}\int_{B_R} \tilde{\mB}(\y,\z)\delta(\y\cdot\z)E_{l}(\y)E_{m}(\z)
  \dd\y\dd\z,
  \quad
  l,m\in\bbZ^d.
\end{equation}
It is easy to check that the coefficients $\hat{B}(l,m)$ are real and satisfy the symmetry relations
\begin{equation} \label{eq:sym_B}
    \hat{B}(l,m) = \hat{B}(m,l) = \hat{B}(l,-m).
\end{equation}
In terms of the Fourier expansion, 
\begin{equation}\label{eq:FGM_Qk}
  \hat{\mQ}^R[\f,\f]_k=\sum_{l,m\in\bbZ^d}\bone(l+m-k) \left( \hat{B}(l,m) -\hat{B}(m,m)
  \right) \hat{\f}_l\hat{\f}_m,
  \quad k\in\bbZ^d.
\end{equation}
Here $\bone(\cdot)$ is the indicator function, equal to $1$ at the origin and $0$ otherwise.

The Fourier Galerkin method (FGM) in the literature (e.g. \cite{pareschi2000Spectral}) starts with a
finite square grid of Fourier modes
\begin{equation}\label{eq:K}
  K := \{k | k=(k_1,\ldots,k_d),-n\le k_1,\ldots,k_d\le n\},
\end{equation}
and the subspace
\begin{equation}
  \bbP_N=\rmspan\left\{ E_k(\bv) | k\in K \right\} \subset L_{\mathrm{per}}^2(\mD_T).
\end{equation}
We shall denote the grid points in $K$ with lower-case letters $j,k,l,m$.
%and $\sum_{j},\sum_{k},\sum_{l}$ and $\sum_{m}$ are understood implicitly to be taken over the set $K$. 
FGM approximates the collision term \eqref{eq:FGM_Qk} by projecting it to the subspace $\bbP_N$
\begin{equation}\label{eq:FGM_Q}
  \hat{Q}^\G_k :=
  \sum_{l,m\in K}\bone(l+m-k) \left( \hat{B}(l,m) -\hat{B}(m,m) \right)
  \hat{\F}_l\hat{\F}_m,
  \quad k\in K,
\end{equation}
where $\hat{F}_k$ for $k\in K$ serve as the approximation of the Fourier modes $\hat{\f}_k$ of the
exact solution.

%% is equivalent to approximating \eqref{eq:FGM_Qv} by
%% \begin{displaymath}
%%   Q^G[f,f](\bv)=\sum_k \hat{Q}_{G,k} E_k(\bv),
%% \end{displaymath}
%% where the subscript $G$ denotes the Galerkin method, and
%% \begin{equation}\label{eq:FGM_Q}
%%     \hat{Q}^G_k=\sum_{\substack{l,m}}\bone(l+m-k)
%%     \left( \hat{B}(l,m) -\hat{B}(m,m) \right)
%%     \hat{\F}_l\hat{\F}_m.
%% \end{equation}

Putting together the above discussion, one arrives at the equations of the discrete Fourier
coefficients $\hat{F}_k(t)$ for $k\in K$ for FGM \cite{pareschi2000Spectral}
\begin{equation}\label{eq:FG}
  \left\{
  \begin{aligned}
    &\od{\hat{\F}_k}{t}= \hat{Q}^\G_k = \sum_{l,m\in K} \bone(l+m-k)
    \left( \hat{B}(l,m)-\hat{B}(m,m) \right)\hat{\F}_l\hat{\F}_m,\\
    &\hat{\F}_k(t=0)=\hat{\F}_k^0,
  \end{aligned}\right.
\end{equation}
where $\hat{\F}_k^0$ are the Fourier coefficients of the initial condition $\f^0(\bv)$ restricted on
$\mD_T$.

\begin{remark}
  The above description of the Fourier Galerkin method is based on the Carleman representation of
  the Boltzmann collision operator \eqref{eq:Carleman}. Starting from the classical form
  \eqref{eq:collision}, one can also derive a relation similar to \eqref{eq:FGM_Q} (see
  \cite{pareschi2000Spectral} for details), while the definition of $\hat{B}(l,m)$ is slightly
  different:
  \begin{equation} \label{eq:classical_hatB}
    \begin{split}
      \hat{B}(l,m) &= \int_{B_{R}} \int_{\bbS^{d-1}} \mB(\bg, \omega) E_l\left(
      \frac{1}{2} (\bg + |\bg| \omega) \right) E_m\left( \frac{1}{2} (\bg - |\bg| \omega)
      \right) \dd\bg \dd\omega,% \\
      %&= \int_{B_{R}} \int_{B_R} \mB(\bg, \bg'/|\bg'|)
      %\delta(|\bg|^d-|\bg'|^d)
      %E_l\left( \frac{1}{2} (\bg + \bg') \right)
      %E_m\left( \frac{1}{2} (\bg - \bg') \right)
      %\dd \bg' \dd \bg,
    \end{split}
  \end{equation}
  where $T\geq\frac{3+\sqrt{2}}{4}R$ in the definition of $E_k(\cdot)$. It is straightforward to
  check that these coefficients also satisfy the symmetry relation \eqref{eq:sym_B}.
\end{remark}

FGM achieves spectral accuracy, although the computational cost is still as high as $O(N^{2d})$
\cite{pareschi2000Spectral}. Two fast algorithms \cite{Mouhot, Hu2016} reduced the cost to
$O(MN^d\log(N))$ for the hard sphere molecules (the Maxwell molecules for 2D case) \cite{Mouhot} and
to $O(MN^{d+1}\log(N))$ for general collision kernels \cite{Hu2016}, where $M$ is the number of
points used for discretizing the unit sphere $\mathbb{S}^{d-1}$.

Compared to DVM, the solution of FGM loses most of the aforementioned physical properties, including
positivity, the conservation of momentum and energy, and the H-theorem. In
\cite{pareschi2000stability}, Pareschi and Russo proposed a positivity preserving regularization of
FGM by using {Fej\'er} filter at the expense of spectral accuracy. Despite this, the solution fails
to satisfy the H-theorem. The loss of the conservation can be fixed by a spectral-Lagrangian
strategy \cite{gamba2009spectral}. 
%However, it remains open how to hybridize these regularizations
%in order to preserve both the positivity and the conservation at the same time.

\subsection{Motivation} 
DVM preserves a number of physical properties (such as positivity of the solution, the H-theorem,
and exact conservation of mass, momentum and energy) but suffers from high computational costs and
low accuracies. FGM enjoys spectral accuracies and lower computational costs but sacrifices almost
all physical properties except the mass conservation. In this paper, we aim for a trade-off between
the physical properties and the spectral accuracy.

As a fundamental property of the solution to the Boltzmann equation, the positivity of the
distribution function helps establish the H-theorem, which is one of the crucial properties in order
to guarantee the well-posedness of the discrete system. Therefore, it makes sense to maintain the
positivity and the H-theorem, as long as it does not sacrifice significantly other properties such
as numerical accuracy and efficiency. This paper is an initial study in this direction.

%% , this paper establishes a filtered spectral method that preserves the positivity, the mass
%% conservation, and the H-theorem. In addition, the computational cost of this new method is the same
%% as the cost of the FGM.

To achieve this goal, we first study carefully the reason behind the loss of the H-theorem in FGM by
comparing it with DVM. With a few novel modifications to FGM, we propose an entropic Fourier method
(EFM) that preserves the positivity, the mass conservation, and the H-theorem. In addition, the
computational cost of this new method is the same as the one of FGM.

The rest of the paper is organized as follows. In Section \ref{sec:result}, we first outline the key
steps to develop EFM and state the main results of the paper. The details of the derivation and some
deeper understandings of the model are provided in Section \ref{sec:model}. Section
\ref{sec:numerical} presents the implementation of EFM and the numerical results. The paper ends
with a discussion in Section \ref{sec:conclusion}.

% \input{result}
% vim: tw=70:spell
\section{Main result}\label{sec:result}
This section outlines the overall procedure of our derivation and lists some key results. Detailed
derivation and investigation will be given in Section \ref{sec:model}.

Aiming at developing an entropic Fourier method for the homogeneous Boltzmann equation, one works
mainly with the evolution of the Fourier coefficients $\hat{\F}_k(t)$. Recall the discrete Fourier
transform
\begin{equation}\label{eq:DFT}
  \F_p=\sum_{k\in K}\hat{\F}_k E_k(p),\qquad  \hat{\F}_k=\frac{1}{N^d}\sum_{p\in\X}\F_p E_{-k}(p),
\end{equation}
where $\X$ defined in \eqref{eq:X} and $K$ defined in \eqref{eq:K} are the sets of uniform samples
in the velocity space and the Fourier domain, respectively.

Using the discrete Fourier transform, one can instead treat the point values $\F_p$ as the degrees
of freedom and write the numerical scheme in the DVM form \eqref{eq:dvmcollision}.  According to the
derivation in Section \ref{sec:DVM}, the following condition is required in order to guarantee the
H-theorem for DVM:
\begin{condition}  \label{cond:DVM}
  The DVM defined in \eqref{eq:dvmcollision} satisfies
  \begin{enumerate}
  \item\label{it:symmetry} the coefficients $A_{pq}^{rs}$ satisfy the symmetry relation
    $A_{pq}^{rs}=A_{qp}^{rs}=A_{rs}^{pq}$;
  \item\label{it:nonnegative} the coefficients $A_{pq}^{rs}$ are non-negative, i.e
    $A_{pq}^{rs}\geq 0$;
  \item\label{it:initialvalue} the initial values are non-negative, i.e $\F_p(t=0)\geq0$ for any
    $p\in\X$.
  \end{enumerate}
\end{condition}

The general idea of our approach is to revise the existing Fourier Galerkin method (FGM) so that
Condition \ref{cond:DVM} is fulfilled. Below we list the steps that lead to a numerical scheme that
satisfies the H-theorem.
\begin{enumerate}
\item
  Apply the Fourier collocation method to \eqref{eq:FGM_Qv}. This leads to an approximation to
  \eqref{eq:FGM_Qk} in the form
  \begin{equation}\label{eq:FGM_Qc}
    \hat{Q}^\C_k = \sum_{l,m\in\K} \bone_N(l+m-k)[\hat{B}_N(l,m) - \hat{B}_N(m,m)]
    \hat{\F}_l\hat{\F}_m,\quad k\in\K,
  \end{equation}
  where $\bone_N(l):=\bone(l \bmod N)$ and $\hat{B}_N(l,m):=\hat{B}(l\bmod N, m\bmod N)$.  Here
  $\bmod$ is the symmetric modulo function, i.e. each component of $l\bmod N$ ranging from $-n$ to
  $n$ (recall $N=2n+1$). Using the relation between the Fourier coefficients and the values on
  collocation points \eqref{eq:DFT}, we can rewrite \eqref{eq:FGM_Qc} as
  \begin{equation}\label{eq:modifiedcollision2}
    Q^\C_r =\sum_{p,q,s\in X}A_{pq}^{rs}[\F_p\F_q-\F_r\F_s],\quad r\in\X,
  \end{equation}
  where $A_{pq}^{rs}$ (given in \eqref{eq:A}) is determined by the Fourier modes of the collision
  kernel $\hat{B}_N(\cdot,\cdot)$ and satisfies the symmetry relation (Condition
  \ref{cond:DVM}.\ref{it:symmetry}).
  %In this process, the collision term of the resulting FGM \eqref{eq:FGM_Q} also takes the following
  %form in the Fourier domain
  %\begin{equation}%\label{eq:modQk}
  %  \hat{\tilde{Q}}_r = \sum_{l,m}
  %  \bone_N(l+m-k)[\hat{B}(l,m) - \hat{B}(m,m)]
  %  \hat{\F}_s\hat{\F}_m,
  %\end{equation}
  %where $\bone_N(l):=\bone(l \bmod N)$.
\item
  A careful study shows that $A_{pq}^{rs}$ fails to be non-negative. This can be fixed by applying a
  positivity preserving filter to $\hat{B}_N(l,m)$, i.e.,
  \begin{equation} \label{eq:Bsigma}
    \hat{B}_N^{\sigma}(l,m):=\hat{B}_N(l,m)\sigma_N(l)\sigma_N(m),
    \quad l,m\in K,
    %\quad \sigma_N(l) =    \prod_{d=1}^d {\sigma}_N(l_d),
  \end{equation}
  where ${\sigma}_N(l)$ is the tensor-product of $d$ one-dimensional modified Jackson filter
  \cite{meinardus1967approximation,weisse2006Filters}.
  %% given by
  %% \begin{equation}\label{eq:JacksonFilter2}
  %%   {\sigma}_N(\beta) = \frac{(m+1-|\beta|)\cos(\frac{\pi|\beta|}{m+1})
  %%     +\sin(\frac{\pi|\beta|}{m+1})\cot(\frac{\pi}{m+1})}{m+1},\quad
  %%   m=\left\lfloor\frac{N-1}{2}\right\rfloor
  %% \end{equation}
  %% and here we overload the notation also for $\sigma(l)$ for $l=(l_1,\ldots,l_d)$.
  The modified collision term takes the following form in the Fourier domain
  \[
  \hat{Q}^\sigma_k=\sum_{l,m\in\K}\bone_N(l+m-k)[\hat{B}_N^\sigma(l,m)-\hat{B}_N^\sigma(m,m)]
  \hat{\F}_l\hat{\F}_m,\quad k\in\K.
  \]
  Using ${(A^\sigma)}^{rs}_{pq}$ to denote the coefficients determined by the new kernel modes
  $\hat{B}_N^{\sigma}(l,m)$ and writing
  \[
  Q^\sigma_r =\sum_{p,q,s\in\X}(A^\sigma)_{pq}^{rs}[\F_p\F_q-\F_r\F_s],\quad r\in\X,
  \]
  one can verify that both the symmetry relation (Condition
  \ref{cond:DVM}.\ref{it:symmetry}) and the non-negativity (Condition
  \ref{cond:DVM}.\ref{it:nonnegative}) are satisfied.
\item
  To guarantee the positivity of the initial values (Condition
  \ref{cond:DVM}.\ref{it:initialvalue}), we adopt interpolation rather
  than orthogonal projection while discretizing the initial
  distribution function.
\end{enumerate}

\paragraph{Main result.}
Summarizing the outline given above, we arrive at a new entropic
Fourier method (EFM) that takes the following simple form
\begin{equation}\label{eq:EFM}
  \left\{  
  \begin{aligned}
    &\od{\hat{\F}_k}{t}= \hat{Q}^\sigma_k = \sum_{l,m\in\K}
    \bone_N(l+m-k)
    \left( \hat{B}_N^\sigma(l,m)-\hat{B}_N^\sigma(m,m) \right)\hat{\F}_l\hat{\F}_m,\\
    &\hat{\F}_k(t=0)= \frac{1}{N^d} \sum_{r\in\X} \f(t=0,r) E_{-k}(r).
  \end{aligned}
  \right.
\end{equation}
This method preserves several key physical properties, as guaranteed by the following theorem.
\begin{theorem}\label{thm:EFM}
  If $\f(t=0,\bv)\geq 0$ for $\bv\in\bbR^d$, then the solution
  $\F_r(t)=\sum_{k\in\K}\hat{\F}_k(t)E_{k}(r)$ for $r\in\X$ of \eqref{eq:EFM} satisfies for $\forall
  t>0$,
  \begin{align}
    \label{eq:conservationMass}
    \text{conservation of mass: } & \od{}{t}\sum_{r\in\X}\F_r(t) = 0, \\
    \label{eq:nonnegativity}
    \text{non-negativity: }  & \F_r(t)\geq0,\quad r\in\X,\\
    %\text{positivity: } & \text{if } $\f(t=0,\bv)>0$ for $\bv\in\bbR$,
    %\text{ then } f(t,\bv_r)>0, \quad \forall t>0, \bv_r\in\X,\\
    \label{eq:Htheorem}
    \text{discrete H-theorem: } &
    \od{}{t}\sum_{r\in\X}\F_r(t)\ln \F_r(t) \leq 0.
  \end{align}
  
\end{theorem}
The proof is presented in Section \ref{sec:proof}. Due to the positivity-preserving filter
\eqref{eq:JacksonFilter2}, the numerical accuracy of EFM in approximating the collision operator is
second order (see Section \ref{sec:birdview} for details).

Another important result for EFM is the existence of fast algorithms. For FGM, the fast
algorithms proposed in \cite{Mouhot, Hu2016} are based on the approximation of the kernel
$\hat{B}_N(\cdot,\cdot)$:
\begin{equation} \label{eq:low_rank_approx}
  \hat{B}_N(l,m)\approx \sum_{t=1}^M
  \alpha_{l+m}^t\beta_{l}^t \gamma_{m}^t,\quad l,m\in K,
\end{equation}
with the number of terms $M \ll N^d$. Since the filtered kernel $\hat{B}_N^{\sigma}(l,m)$ turns out
to have a similar approximation
\begin{equation} \label{eq:low_rank_approx_filtered}
  \hat{B}_N^{\sigma}(l,m)\approx \sum_{t=1}^M 
  \alpha_{l+m}^t
  \left(\sigma_N(l)\beta_l^t\right)
  \left(\sigma_N(m)\gamma_m^t\right),
\end{equation}
these fast algorithms still apply. Moreover, when the above approximation is applied, the H-theorem
still holds. Detailed discussion will be given in Section \ref{sec:implementation}.

% \input{model}
% vim: tw=70:spell
\section{Entropic Fourier method}\label{sec:model}
As shown in Section \ref{sec:DVM}, a discrete H-theorem can be obtained from the classical DVM,
where the associated entropy function can be considered as a numerical quadrature for the integral
of $f \ln f$. This requires the positivity of the distribution function, which can be guaranteed by
the positivity of the discrete collision kernel $A_{pq}^{rs}$. In general, to preserve the Boltzmann
entropy in the numerical scheme, the positivity of the numerical solution needs to be enforced in a
certain sense due to the presence of $\ln f$ in the entropy function. However, in FGM, there is no
guarantee of any form of positivity in the numerical solution, and hence the H-theorem does not
hold.

%% In \cite{pareschi2000stability}, the authors proposed a positivity preserving Fourier
%% spectral method by introducing strong filters, so that the distribution function is point-wise
%% non-negative and the approximation of the Boltzmann entropy becomes possible. However, the numerical
%% solution of this approach turns out to be highly dissipative, and whether the H-theorem holds is
%% still unclear.

In this paper, rather than enforcing the non-negativity of the whole distribution function, we take
a collocation approach and focus on the non-negativity only at the collocation points. Based on this
idea, we start from a collocation method for the homogeneous Boltzmann equation and write it as a
DVM of the function values defined at the collocation points. One then tries to alter the
coefficients to match the requirements in Condition \ref{cond:DVM} so that the H-theorem can be
subsequently derived.

The three steps listed in Section \ref{sec:result} are elaborated in the first three subsections
below.  After that, Section \ref{sec:birdview} compares the entropic Fourier method (EFM) with
other Fourier methods.

\subsection{Fourier collocation method in a DVM form} \label{sec:sym_FCM}
%% In order to derive a H-theorem, our basic idea is to write a Fourier method into a DVM form
%% \eqref{eq:dvmcollision} and then revise it to satisfy Condition \ref{cond:DVM}. 
The mechanisms of DVM and FGM are quite different: DVM is concerned with the values of the
distribution on discrete points, whereas FGM \eqref{eq:FGM_Q}, as a Galerkin method, works on the
Fourier modes of the distribution function. It is not straightforward how to link these two methods.
Alternatively, we will consider another type of Fourier methods --- the collocation method (also
known as the pseudospectral method).

\subsubsection{Fourier collocation method}
In the Fourier collocation method (FCM), the collision term on the set $X$ is evaluated directly
using \eqref{eq:FGM_Qv}:
\begin{equation}\label{eq:FGM_Qcv}
  Q^\C_r=\sum_{l,m\in\K}\left( \hat{B}_N(l,m)-\hat{B}_N(m,m) \right)\hat{\F}_l\hat{\F}_m E_{l+m}(r),
  \quad r\in \X,
\end{equation}
where $\hat{B}_N(l,m):=\hat{B}(l\bmod N, m\bmod N)$, and $\bmod$ is the symmetric modulo function,
i.e. each component of $l\bmod N$ ranging from $-n$ to $n$ (recall $N=2n+1$). Since the above
equation only uses the value of $\hat{B}_N$ in $K$, one can use $\hat{B}$ and $\hat{B}_N$
interchangeably here. Here we note that $\hat{B}_N(l,m)$ satisfy the symmetry relation
\eqref{eq:sym_B}.

The corresponding Fourier modes can be obtained by an inverse discrete Fourier transform:
\begin{equation}\label{eq:FCM}
  \begin{aligned}
    \hat{Q}^\C_k &= \frac{1}{N^d}\sum_{r\in\X} Q_r E_{-k}(r)\\
    &= \sum_{l,m\in\K}\bone_N(l+m-k)\left( \hat{B}_N(l,m)-\hat{B}_N(m,m)\right)\hat{\F}_l\hat{\F}_m,
  \end{aligned}
\end{equation}
where $\bone_N(l):=\bone(l \bmod N)$.

If the initial value is smooth enough, due to the smoothing effect of the Boltzmann collision
operator \cite{Barbaroux2017}, both the Fourier Galerkin and the Fourier collocation methods
have spectral accuracy \cite{parischi2014numerical}. Moreover, in some cases, FCM \eqref{eq:FCM} is
numerically more efficient, especially for the fast summation algorithms in \cite{Mouhot, Hu2016}.
For example, in \cite{Mouhot}, the following approximation of $\hat{B}_N(l,m)$ is considered:
\begin{equation}
  \hat{B}_N(l,m) \approx \sum_{t=1}^M \beta^t_l\gamma_m^t,
\end{equation}
where $M\in\bbN^+$ is the total number of quadrature points on the sphere. Then the collision term
in this Galerkin method can be approximated by
\begin{equation}\label{eq:fastFGM}  %hat{Q}_r\approx
  \sum_{t=1}^M
  \sum_{l,m\in K}
  \bone(l+m-k)\left[ \left(\beta^t_l\hat{\F}_l\right) 
    \left(\gamma^t_m\hat{\F}_m\right) 
    - \hat{\F}_l\left(\beta^t_m\gamma^t_m\hat{\F}_m\right)\right].
\end{equation}
To evaluate \eqref{eq:fastFGM} efficiently, one needs to utilize FFT-based convolutions. To obtain
these coefficients, one needs the zero-padding technique to avoid aliasing. If one uses the same
method to evaluate $\hat{Q}_k$ in \eqref{eq:FCM}, then {\em no zero padding is needed}. Therefore
the collocation method shortens the length of vectors used in the Fourier transform, which makes the
algorithm faster.

\subsubsection{DVM form}
To link FCM with DVM, we split the collision term \eqref{eq:FGM_Qcv} into the gain part ($+$) and
the loss part ($-$):
\begin{equation}
  Q^{\C,+}_r = \sum_{l,m\in\K}\hat{B}_N(l,m)\hat{\F}_l\hat{\F}_m E_{l+m}(r), \quad
  Q^{\C,-}_r = \sum_{l,m\in\K}\hat{B}_N(m,m)\hat{\F}_l\hat{\F}_m E_{l+m}(r),\quad r\in\X.
\end{equation}
Noticing $E_{l+m}(r)=\sum_{k\in\K}\bone_N(l+m-k)E_{k}(r)$ and plugging \eqref{eq:DFT} into the gain part
yields
\begin{equation}\label{eq:FGM_Qcv+}
  Q^{\C,+}_r = \frac{1}{N^{2d}}\sum_{\substack{l,m,k\in\K\\p,q\in\X}} \bone_N(l+m-k)\hat{B}_N(l,m)
  E_{-l}(p)E_{-m}(q)E_{k}(r)\F_p\F_q.
\end{equation}
%% \begin{equation}\label{eq:Ejs}
%%   \frac{1}{N^d}\sum_{s\in\X}E_{j}(s)= \bone_N(j),
%% \end{equation}
Since $\frac{1}{N^d}\sum_{s\in\X}E_{j}(s)= \bone_N(j)$, one can sum over $j$ to get
$\frac{1}{N^d}\sum_{j\in K,s\in\X}E_j(s)=1$. With this equation, one introduces two new indices to
\eqref{eq:FGM_Qcv+} by multiplying its right hand side with 
$\frac{1}{N^d} \sum_{j\in K,s\in X}E_j(s)$:
\begin{small}
  \begin{equation}\label{eq:Qc+}
      Q^{\C,+}_r = \frac{1}{N^{3d}} \sum_{\substack{l,m,k,j\in\K\\p,q,s\in\X}} \bone_N(l+m-k-j)
      \hat{B}_N(l-j,m-j) E_{-l}(p)E_{-m}(q)E_{k}(r)E_{j}(s)\F_p\F_q.
  \end{equation}
\end{small}%
%% where it is also used that the summation above with respect to $s$ is nonzero only when $j = 0$ due
%% to \eqref{eq:Ejs}.
%% In \eqref{eq:FCM}, only $\hat{B}(l,m)$ with each component of $l,m$ ranging from $-n$ to $n$ are
%% used (with $N=2n+1$), thus one can extend the definition of $\hat{B}$ periodically in all the
%% dimensions, and still denote it by $\hat{B}$, i.e., $\hat{B}(l-j,m-j) = \hat{B}((l-j) \bmod N, (m-j)
%% \bmod N)$.
If one introduces
\begin{equation} \label{eq:A}
  A_{pq}^{rs}=\frac{1}{N^{3d}}\sum_{l,m,k,j\in\K}\bone_N(l+m-k-j)\hat{B}_N(l-j,m-j)E_{-l}(p)E_{-m}(q)E_{k}(r)E_{j}(s),
\end{equation}
then the gain term is
\begin{equation}\label{eq:gainterm1}
  Q^{\C,+}_r =\sum_{p,q,s\in\X}A_{pq}^{rs}\F_p\F_q.
\end{equation}
Apparently, such term does take the form of the gain term of DVM \eqref{eq:dvmcollision}.

For the loss term, the identity
\begin{equation}\label{eq:def_A}
  \sum_{p,q\in\X}A_{pq}^{rs}=\frac{1}{N^d}\sum_{k,j}\bone_N(k+j)\hat{B}_N(j,j)E_{k}(r)E_{j}(s),
\end{equation}
leads to the following derivation
\begin{equation}
  \begin{aligned}
    \sum_{p,q,s\in\X}A_{pq}^{rs}\F_r\F_s
    &= \frac{1}{N^d} \sum_{k,j\in K,s\in X}\bone_N(k+j)\hat{B}_N(j,j)E_{k}(r)E_{j}(s) \F_r\F_s\\
    &= \frac{1}{N^d}\sum_{k,j,l,m\in K,s\in X}\!\bone_N(k+j)\hat{B}_N(j,j)E_{k}(r)E_{j}(s)E_{l}(r)E_{m}(s)
    \hat{\F}_l\hat{\F}_m\\
    &=\sum_{k,j,l,m\in K}\bone_N(k+j)\bone_N(j+m)\hat{B}_N(j,j)\hat{\F}_l\hat{\F}_m E_{k+l}(r)\\
    &=\sum_{l,m\in\K}\hat{B}_N(m,m)\hat{\F}_l\hat{\F}_m E_{l+m}(r) = Q^{\C,-}_r.
  \end{aligned}
\end{equation}
%The second last equality holds because $\bone_N(j+m)\bone_N(k+j)$ indicates that the summand is
%nonzero only if $j=-m\bmod N$ and $k=m\bmod N$, which implies that 
%$\sum_{k,j}\bone_N(j+m)\bone_N(k+j)\hat{B}_N(j,j)=\hat{B}_N(-m\bmod N,-m\bmod N)=\hat{B}_N(m,m)$.

%We note that the last second ``='' is not trivial for the case $N$ is even. Precisely, we take
%$D=1$ as an example and $m$ ranges from $-N / 2$ to $N / 2 - 1$, which is not symmetry with respect
%to the origin.  If $m= -N / 2$, then $\bone_N(j+m)\bone_N(k+j)$ indicates $k=j=m=-N / 2$, the
%equality holds due to the symmetry relation of $\hat{B}(l,m)$ \eqref{eq:sym_B}.
In summary, one can write FCM in the following DVM form
\begin{equation} \label{eq:QCvr}
  Q^\C_r = \sum_{p,q,s\in X}A_{pq}^{rs}(\F_p\F_q-\F_r\F_s)
\end{equation}
with $A_{pq}^{rs}$ given in \eqref{eq:A}. Finally, the symmetry relation (Condition
\ref{cond:DVM}.1)
\begin{equation}\label{eq:symA}
  A_{pq}^{rs} = A_{qp}^{rs} = A_{rs}^{pq}
\end{equation}
holds, as this can be easily seen by the symmetry relation of $\hat{B}_N(l,m)$ and switching the
indices in \eqref{eq:A}.

%% The positivity of these coefficients will be discussed in the next subsection.

%If $N$ is odd, then each component of $l,m,j,k$ ranges from $-(N-1)/2$
%to $(N-1)/2$. The relation \eqref{eq:symA} follows from the result of
%the symmetry of $\hat{B}(l,m)$ in \eqref{eq:sym_B}. If $N$ is even,
%then each component of $l,m,j,k$ ranges from $-N/2$ to $N/2-1$.
%Consequently, for each $l$, there may not exist a $m$ such that
%$m=-l$. Thus, the relation \eqref{eq:symA} does not hold. However,
%the symmetry can be recovered by forcing $\hat{B}(l,m)$ to be zero
%when
%\begin{displaymath}
%\min(l_1,\cdots,l_d, m_1,\cdots,m_d)=-N/2,
%\end{displaymath}
%so that these coefficients do not take effect.
%applying a filter $\sigma_N(l,m)$ to
%$\hat{B}(l,m)$, so that the kernel $\hat{B}(l,m)$ is replaced by
%$\sigma_N(l,m)\hat{B}(l,m)$. If $\sigma_N(l,m)$ satisfies
%\begin{small}
%\begin{equation} \label{eq:evenN_condition}
%    \sigma_N(l,m) = \sigma_N(m,l) = \sigma_N(-l,m), \quad 
%    \sigma_N(j,k) = 0 \text{ if } \min(j_1, \cdots,j_d,k_1, \cdots,k_d)=-N/2,
%\end{equation}
%\end{small}%
%then \eqref{eq:symA} holds again. In this case, the accuracy of the
%spectral method depends on the filter.

\subsection{Positivity preserving} \label{sec:filteredFGM}
As remarked earlier, in order to obtain an H-theorem for FCM, one needs to ensure that all the
coefficients $A_{pq}^{rs}$ are non-negative (Condition \ref{cond:DVM}.\ref{it:nonnegative}). Below,
we first show that $A_{pq}^{rs}$ as defined in \eqref{eq:QCvr} fail to be non-negative, and then
apply a filter to recover non-negativity.

\subsubsection{Failure of positivity preservation in FCM}
\label{sec:failure}
We start by simplifying the coefficients $A_{pq}^{rs}$ based on
\eqref{eq:A}
\begin{equation}\label{eq:modifiedA}
  \begin{aligned}
    A_{pq}^{rs}&=
    \frac{1}{N^{3d}}\sum_{l,m,k,j\in\K}\bone_N(l+m-k-j) \hat{B}_N(l-j,m-j)E_{-l}(p)E_{-m}(q)E_k(r)E_j(s)\\
    &= \frac{1}{N^{3d}}\sum_{l,m,k\in\K} \hat{B}_N(m-k, l-k) E_{-l}(p-s) E_{-m}(q-s) E_{k}(r-s).
  \end{aligned}
\end{equation}
where one uses $j=l+m-k \bmod N$. By performing a change of variables $i=(m-k)\bmod N$ and
$j=(l-k)\bmod N$, we arrive at
\begin{equation} \label{eq:A_simplified}
  \begin{aligned}
    A_{pq}^{rs} &=
    \frac{1}{N^{3d}}\sum_{i,j,k\in\K}\hat{B}_N(i,j) E_{-k-i}(p-s) E_{-k-j}(q-s) E_k(r-s)\\
    &=\bone_N(r+s-p-q)\frac{1}{N^{2d}} \sum_{i,j\in\K}\hat{B}_N(i,j)E_{-i}(p-s)E_{-j}(q-s).
  \end{aligned}
\end{equation}
%Thus, we have $A_{rs}^{pq}=A_{r-s,0}^{p-s,q-s}$. Since
%\begin{equation}\label{eq:tildeA}
%    \begin{aligned}
%        A_{r0}^{pq} &= \frac{1}{N^{3d}}\sum_{l,m,k}
%        \hat{B}(m-k, l-k) E_{-l}(p)E_{-m}(q)E_{k}(r)\\
%        &=\frac{1}{N^{3d}}\sum_{i,j,k}\hat{B}(i,j)
%        E_{-k-i}(p) E_{-k-j}(q) E_k(r)\\
%        &=\bone_N(r-p-q)\frac{1}{N^{2d}}
%        \sum_{i,j}\hat{B}(i,j)E_{-i}(p)E_{-j}(q),
%    \end{aligned}
%\end{equation}
%we arrive at
%\begin{equation} \label{eq:tildeA_G}
%    A_{rs}^{pq}=\bone_N(r+s-p-q)\frac{1}{N^{2d}}\sum_{l,m}
%    \hat{B}(l,m)E_{-l}(p-s)E_{-m}(q-s).
%\end{equation}
By introducing
\begin{equation} \label{eq:G}
  G(\y, \z)=\sum_{i,j\in\K}\hat{B}_N(i,j)E_{-i}(\y)E_{-j}(\z),\quad \y,\z\in\mD_T,
\end{equation}
which is by definition a periodic function with period $\mD_T$, one can write compactly
\begin{equation}\label{eq:A_G}
  A_{pq}^{rs}=\frac{1}{N^{2d}}\bone_N(r+s-p-q)G(p-s, q-s),
  \quad p,q,r,s\in\X.
\end{equation}

In order to check whether $A_{pq}^{rs}$ is non-negative, one just needs to check whether
$G(\cdot,\cdot)$ is non-negative on the collocation points in $\X$. To get a better understanding of
the function $G(\cdot,\cdot)$ as defined in \eqref{eq:G}, one applies the definition of
$\hat{B}_N(\cdot,\cdot)$ to obtain
\begin{equation}\label{eq:defG}
  G(\y,\z)=\int_{B_R}\int_{B_R}\tilde{\mB}(\y',\z')\delta(\y'\cdot\z')\chi_N(\y-\y')\chi_N(\z-\z')\dd\y'\dd\z'.
\end{equation}
Here $\chi_N$ is the Dirichlet kernel over $\mD_T$ defined by
\begin{equation}
  \chi_N(\bv) = \sum_{k\in\K}E_k(\bv),\quad v\in\mD_T,
\end{equation}
and its discrete Fourier transform $\hat{\chi}_N(k)$ is equal to 1 for $k\in\K$ and 0 on
$\bbZ^d\setminus\K$. By introducing a periodic function in $\mD_T$
\begin{equation} \label{eq:Delta}
  H(\y,\z) = \tilde{\mB}(\y,\z)\delta(\y\cdot\z)\bone(|\y|\leq R) \bone(|\z|\leq R), \quad
  \y,\z\in\mD_T,
\end{equation}
one can write 
\[
G = H * (\chi_N\otimes\chi_N),
\]
where the convolution is defined periodically in $\mD_T\times\mD_T$. Equivalently, $G(\y,\z)$ is also
the truncated Fourier expansion of $H(\y,\z)$ by keeping only the frequencies in $\K$.

Although $H(\y,\z)$ is non-negative in the weak sense, its truncated Fourier approximation
$G(\y,\z)$ fails to be so. For example, the values of $G$ for the kernel
$\tilde{\mB}(\y,\z)\equiv\frac{1}{\pi}$, $R=6$ in 2D are plotted in Figure \ref{fig:G}. This clearly
shows that negative values appear as expected. Therefore in general, the H-theorem does not hold for
FCM.

\begin{figure}[!ht]
  \centering
  \includegraphics[width=.5\textwidth]{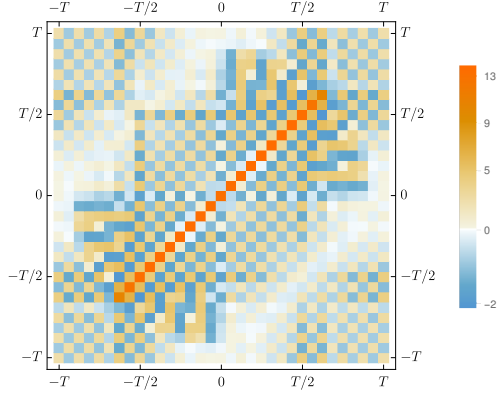}
  \caption{The values of $G(\y, \z)$ with $N=32$ at $\z = (T/2,T/2)$. The axes are the two components
    of $\y$.
  }
\label{fig:G}
\end{figure}

%\begin{remark}
%    Since $\mL(f)$ is bounded if $\int_{\bbR^d}f\dd\bv$ is bounded, to
%    preserve $\F(\bv)$ being positive at the points $r$, we
%    just need $\sum_{s}\tilde{A}_{rs}^{pq}\geq 0$. Noticing
%    \eqref{eq:Ejs}, we have 
%    \begin{equation}
%        \begin{aligned}
%            \sum_{s}\tilde{A}_{rs}^{pq} &= \frac{1}{N^{2d}}\sum_{l,m,k}
%            \bone_N(l+m-k)\hat{B}(l,m)E_{-l}(p)E_{-m}(q)
%            E_{k}(r)\\
%            &= \frac{1}{N^{2d}}\sum_{l,m}
%            \hat{B}(l,m)E_{-l}(p-r)E_{-m}(q-r)\\
%            &=G(p-r,q-r).
%        \end{aligned}
%    \end{equation}
%    That is to say, $G(p,q)\geq0$ for any $p,q$ is the
%    sufficient and necessary condition of preserving non-negativity of
%    $\\F_r$ with non-negativity initial value.
%\end{remark}

\subsubsection{Filtering} \label{sec:SymFGM2}
In the previous subsection, one can see that if $\chi_N(\cdot)$ were a non-negative function, then
$G(\y, \z)$ would be non-negative for any $\y$ and $\z$. Thus, in order to get non-negative
coefficients, a possible way is to replace the function $\chi_N$ by a non-negative one. Note that
$\chi_N(\bv)$ is Dirichlet kernel, which is an approximation of the Dirac delta. As $N
\rightarrow +\infty$, the function $\chi_N(\cdot)$ tends to Dirac delta weakly in an oscillatory
way. As pointed out in \cite{filbet2011analysis}, the oscillation breaks the non-negativity of
the solution.

As mentioned earlier in Section \ref{sec:result}, we adopt the one-dimensional modified Jackson
filter \cite{meinardus1967approximation, weisse2006Filters} given by
\begin{equation}\label{eq:JacksonFilter2}
  {\sigma}_N(\beta) = \frac{(n+1-|\beta|)\cos(\frac{\pi|\beta|}{n+1})
    +\sin(\frac{\pi|\beta|}{n+1})\cot(\frac{\pi}{n+1})}{n+1},\quad
\end{equation}
where $N=2n+1$ and $-n \le \beta\le n$. By a slight abuse of notation, the $d$-dimensional modified
Jackson filter for a multiindex $k=(k_1,\ldots,k_d)\in\K$ is defined through tensor-product
\[
\sigma_N(k) =    \prod_{i=1}^d {\sigma}_N(k_i).
\]
The modified kernel $\chi^\sigma_N(v)$ can then be defined as
\[
\chi^\sigma_N(\bv) = \sum_{k\in\K}\sigma_N(k) E_k(\bv),\quad v\in\mD_T.
\]

%If $N$ is even, one can define
%\begin{equation}\label{eq:JacksonFilterEven}
%    \bar{\sigma}_{N}(n)=\bar{\sigma}_{N-1}(n),\quad
%    n=-N/2+1,\cdots,N/2-1,\quad 
%    \bar{\sigma}_N(\pm N/2)=0.
%\end{equation}
%Since $\bar{\sigma}_N(N/2) = 0$, one also gets $\chi_N^{\sigma}(\bv)
%\ge 0$ with the same method.

Once $\chi_N$ is replaced with $\chi_N^{\sigma}$ in \eqref{eq:defG}, the function $G(\y,\z)$ is
substituted with $G^\sigma(\y,\z):=(G*(\chi^\sigma_N\otimes\chi^\sigma_N))(\y,\z)$. A direct
calculation shows that
\begin{small}
  \begin{equation}\label{eq:tildeG}
    \begin{aligned}
      G^\sigma(\y,\z)
      &=\frac{1}{N^{2d}} \int_{B_R}\int_{B_R}\tilde{\mB}(\y',\z')\delta(\y'\cdot\z')
      \chi_N^\sigma(\y-\y')\chi_N^\sigma(\z-\z')\dd\y'\dd\z'\\
      &=\frac{1}{N^{2d}}\sum_{l,m\in\K} \int_{B_R}\int_{B_R}\tilde{\mB}(\y',\z')\delta(\y'\cdot\z')
      \sigma_N(l)\sigma_N(m) E_l(\y-\y')E_m(\z-\z')
      \dd\y'\dd\z'\\
      &=\frac{1}{N^{2d}}\sum_{l,m\in\K}\left[\sigma_N(l)\sigma_N(m)
        \hat{B}_N(l,m)\right]E_{-l}(\y)E_{-m}(\z)\\
      &=\frac{1}{N^{2d}}\sum_{l,m\in\K}\hat{B}_N^{\sigma}(l,m) E_{-l}(\y)E_{-m}(\z),
    \end{aligned}
\end{equation}
\end{small}%
where $\hat{B}_N^{\sigma}(l,m):=\hat{B}_N(l,m)\sigma_N(l)\sigma_N(m)$ as defined in
\eqref{eq:Bsigma}. With $G^\sigma(\y,\z)\ge 0$ guaranteed, one can mimic \eqref{eq:A_G} and define
\begin{equation} \label{eq:tildetildeA}
  (A^\sigma)^{rs}_{pq}=\frac{1}{N^{2d}} \bone_N(r+s-p-q)G^\sigma(p-s,q-s),
  \quad p,q,r,s\in\X,
\end{equation}
which are apparently non-negative. Since replacing $\hat{B}_N(l,m)$ with $\hat{B}_N^{\sigma}(l,m)$ does
not affect the symmetry relation \eqref{eq:sym_B}, the new coefficients $(A^\sigma)^{rs}_{pq}$ also
satisfies $(A^\sigma)^{rs}_{pq}=(A^\sigma)^{rs}_{qp}=(A^\sigma)^{pq}_{rs}$.

From the above discussion, we now define the entropic collision term to be
\begin{equation}\label{eq:Qsk}
  \hat{Q}^\sigma_k := \sum_{l,m\in\K} \bone_N(l+m-k)
  \left(\hat{B}_N^\sigma(l,m)-\hat{B}_N^\sigma(m,m) \right)\hat{\F}_l\hat{\F}_m,
  \quad k\in\K
\end{equation}
in the Fourier domain. In the velocity domain, it is equal to
\begin{equation}\label{eq:Qsr}
  Q^\sigma_r := \sum_{p,q,s\in\X} (A^\sigma)^{rs}_{pq}\left(\F_p\F_q-\F_r\F_s \right),
  \quad r\in\X.
\end{equation}

%=============
\subsection{Initial condition}

In order to obtain the H-theorem, one needs to make sure that the initial data are
non-negative. Since the collision term in \eqref{eq:Qsr} depends only on the values at the
collocation points in $X$, it is sufficient to make the initial data non-negative at these
collocation points. Consequently, it is natural to use sampling rather than orthogonal projection
while preparing the discrete initial data $\F^0_r$ for $r\in\X$. More precisely,
\begin{equation}\label{eq:I_N}
    F^0_r = \mI_N f^0:=
    \begin{cases}
        \f^0(r), & f^0\;\;\text{is continuous},\\
        (\varphi^{\epsilon}*\f^0)(r), & \text{otherwise},
    \end{cases}
\end{equation}
where $\varphi^{\epsilon}\geq0$ is a mollifier, such that $\|\f^0 -
\varphi^{\epsilon}*\f^0\|_{L^2}<\epsilon$ for $\epsilon$ sufficiently small. Once $\{F^0_r\}$ are
ready, the corresponding Fourier coefficients $\{\hat{\F}^0_k\}$ are computed via a fast Fourier
transform.

At this point, all ingredients of the entropic Fourier method (EFM) are ready. The Cauchy problem of
EFM takes the following form as a DVM
\begin{equation}\label{eq:modifiedIVP}
  \left\{
  \begin{aligned}
    &\od{\F_r}{t} = Q^\sigma_r = \sum_{p,q,s\in\X} (A^\sigma)^{rs}_{pq}\left(\F_p\F_q-\F_r\F_s \right),\\
    &\F_r(t=0) = \F^0_r,
  \end{aligned}
  \right.
\end{equation}
Equivalently in the Fourier domain, EFM takes the form
\begin{equation}\label{eq:EFMk}
  \left\{  
  \begin{aligned}
    &\od{\hat{\F}_k}{t}=\hat{Q}^\sigma_k= \sum_{l,m\in\K} \bone_N(l+m-k)
    \left(\hat{B}_N^\sigma(l,m)-\hat{B}_N^\sigma(m,m)\right)\hat{\F}_l\hat{\F}_m,\\
    &\hat{\F}_k(t=0)=\hat{\F}^0_k.
  \end{aligned}
  \right.
\end{equation}

\begin{remark}
  The same technique can be applied to the Fourier method derived from the classical form of the
  Boltzmann collision operator \eqref{eq:collision} to obtain an entropic Fourier method.  In fact,
  from \eqref{eq:classical_hatB} and following the definition of $G(\y,\z)$ in \eqref{eq:G}, one can
  directly obtain
  \begin{equation*}
      G(\y,\z) = (H * (\chi_N\otimes\chi_N))(\y,\z),\quad
      H(\bg,\bg') = \mB(\bg,\omega)\delta(|\bg|-|\bg'|) \bone(|\bg|\leq R), \quad \bg,\bg'\in\mD_T.
      %G(\y,\z)&=
      %\int_{B_R}\int_{S^{d-1}}H(\bg,\omega)
      %\chi_N\left(\frac{\bg+|\bg|\omega}{2}-\y\right)
      %\chi_N\left(\frac{\bg-|\bg|\omega}{2}-\z\right)\dd\bg\dd\omega,
      %%\label{eq:classical_tildeG}
  \end{equation*}
  Again, by \eqref{eq:A_G}, the positivity of $A_{pq}^{rs}$ depends only on the positivity
  of $G(\y, \z)$ at the collocation points. Therefore, replacing $\chi_N$ with $\chi_N^{\sigma}$
  does the job.
\end{remark}

\subsection{Comparison}\label{sec:birdview}

The derivation of EFM switched frequently between the language of Fourier methods and DVM for
different purposes. As we have shown in \eqref{eq:modifiedIVP} and \eqref{eq:EFMk}, EFM can be
regarded either as a Fourier method or as a special DVM.

%% Below we are going to revisit EFM from these two points of view.
%% \subsubsection{EFM as a spectral method}\label{sec:FGMs}
%% We first review several aforementioned versions of FGMs.

In what follows, we provide a comparison between the entropic Fourier method (EFM), the FGM (Fourier
Galerkin method), and the Fourier collocation method (FCM). To set up a uniform notation, let
$\mathcal{Q}[\cdot;\cdot, \cdot]$ be the general collision operator
\begin{equation}
  \mathcal{Q}[C; f, f](\bv) = \int_{\mD_T} \int_{\mD_T} C(\y, \z) [f(\bv+\y) f(\bv+\z) - f(\bv)
    f(\bv+\y+\z)] \dd \y \dd \z
\end{equation}
with a collision kernel $C(\cdot, \cdot)$. Thus the truncated collision term
\eqref{eq:approximated_collision} can be written as $\mathcal{Q}(H; f, f)$ using the definition of
$H$ in \eqref{eq:Delta}. 

Notice that a special feature of a function in $\bbP_N$ is that it is uniquely defined via its
function values at points in $X$ defined in \eqref{eq:X}. Therefore, $\{\F_p|p\in\X\}$ can be
regarded both as discrete set of values or the samples from the smooth periodic $\F(v)\in \bbP_N
\subset L_{\mathrm{per}}^2(\mD_T)$.  By introducing two operators
\[
\mP_N: f \rightarrow \chi_N*f,\quad
\mS_N^\sigma: f \rightarrow \chi_N^\sigma * f
\]
for the space $L_{\mathrm{per}}^2(\mD_T)$, the three methods are different approximations of
$\mathcal{Q}(H; f,f)$ with different initial values: 

\begin{align}
  \label{eq:FGM1}
  \text{FGM: } &
  \mP_N \mathcal{Q}[(\mP_N\otimes\mP_N) H; F, F],
  & &F(t=0,\bv)=\mP_N \f(t=0,\bv), \\
  \label{eq:FGM3}
  \text{FCM: } &
  \underline{\mI_N} \mathcal{Q}[(\mP_N\otimes\mP_N) H; F, F],
  & & F(t=0,\bv)=\mP_N \f(t=0,\bv), \\
  %% \label{eq:FGM4}
  %% \text{Filtered FCM: } &
  %% \mI_N \mathcal{Q}[\mS_N^{\sigma} \mP_N K; f, f],
  %% & & f(t=0,\bv)=\mP_N \f(t=0,\bv), \\
  \label{eq:FGM5}
  \text{EFM: } &
  \mI_N \mathcal{Q}[\underline{(\mS_N^{\sigma}\otimes\mS_N^\sigma)} H; F, F],
  & & F(t=0,\bv)=\underline{\mI_N} \f(t=0,\bv).
\end{align}
%\begin{align}
%    \label{eq:FGM1}
%    \mP_N \mathcal{G}^+(\mP_N K; f, f) &+
%    \mP_N \mathcal{G}^-(\mP_N K; f, f),
%    & f(t=0,\bv)&=\mP_N F(t=0,\bv), \\
%    \label{eq:FGM2}
%    \mP_N \mathcal{G}^+(\mP_N K; f, f) &+
%    \mI_N \mathcal{G}^-(\mP_N K; f, f),
%    & f(t=0,\bv)&=\mP_N F(t=0,\bv), \\
%    \label{eq:FGM3}
%    \mI_N \mathcal{G}^+(\mP_N K; f, f) &+
%    \mI_N \mathcal{G}^-(\mP_N K; f, f),
%    & f(t=0,\bv)&=\mP_N F(t=0,\bv), \\
%    \label{eq:FGM4}
%    \mI_N \mathcal{G}^+(\mS_N^{\sigma} K; f, f) &+
%    \mI_N \mathcal{G}^-(\mS_N^{\sigma} K; f, f),
%    & f(t=0,\bv)&=\mP_N F(t=0,\bv), \\
%    \label{eq:FGM5}
%    \mI_N \mathcal{G}^+(\mS_N^{\sigma} K; f, f) &+
%    \mI_N \mathcal{G}^-(\mS_N^{\sigma} K; f, f),
%    & f(t=0,\bv)&=\mI_N F(t=0,\bv).
%\end{align}
where $\mI_N$ is the interpolation operator defined in 
\eqref{eq:I_N}. 
%associated with the grid $\X$ for continuous functions in the space $L_{\mathrm{per}}^2(\mD_T)$.

%% Here $\mP_N H:= H*(\chi_N\otimes\chi_N)$ or equivalently
%% \[
%% (\mP_N H)(\bp,\bq) = \frac{1}{N^{2d}} \sum_{k,l} E_k(\bp) E_l(\bq) \int_{\mD_T} \int_{\mD_T}
%% H(\y,\z) E_{-k}(\y) E_{-l}(\z) \dd\y \dd\z,
%% \]
%% and $\mS_N^{\sigma}$ is considered as the filter for the $2D$-dimensional function with
%% \[
%% \mS_N^\sigma\mP_N H = H*(\chi^\sigma_N\otimes\chi^\sigma_N).
%% \]

The list \eqref{eq:FGM1} to \eqref{eq:FGM5} clearly shows how we change from FGM to EFM in our
derivation. The last line \eqref{eq:FGM5} also shows that EFM provides an approximation of the
original binary collision operator in the language of spectral methods. Below we will briefly review
the basic properties of all the three methods.

The method \eqref{eq:FGM1} stands for FGM as described in Section \ref{sec:FGM}. In the
derivation, the kernel $K$ is not explicitly projected. However, \eqref{eq:FGM_Q} shows that the
discrete collision operator depends only on $(\mP_N\times\mP_N) H$. By replacing the projection
operator applied to $\mathcal{Q}$ with interpolation as in \eqref{eq:FCM}, we arrive at the Fourier
collocation method \eqref{eq:FGM3} introduced in Section \ref{sec:sym_FCM}.  Since a direct
projection of $H$ does not preserve the positivity of the kernel, the negative part of the discrete
kernel may cause a violation of the H-theorem.  Nevertheless, both these two methods have spectral
accuracy in the velocity space.

To ensure the positivity of the discrete kernel, the filter $\mS_N^{\sigma}\otimes\mS_N^{\sigma}$ is
applied in \eqref{eq:FGM5}, and thus positive coefficients \eqref{eq:tildetildeA} are obtained. The
method \eqref{eq:FGM5} also ensures the positivity of the approximation of $\F$ at collocation
points, and thus the discrete H-theorem follows.

However, the filter $S_N^{\sigma}$ has a smearing effect, which reduces the order of convergence.
For any smooth periodic function $f\in L_{\mathrm{per}}^2(\mD_T)$, the $L^2$-error
$\|f-\mS_N^{\sigma}f\|_2$ is $O(N^{-2})$ \cite[Chapter 4]{lorentz1966approximation}, and therefore EFM is at most
second-order. On the other hand, if one splits the collision term of EFM into the gain part and the
loss part and let $\F=\mP_N f$, then
\begin{align}
  Q^{\sigma,+}[\F,\F](r) &= \sum_{l,m\in K}\hat{B}_N(l,m)  \sigma_N(l)\hat{\F}_l  \sigma_N(m)\hat{\F}_m 
  E_{l+m}(r) = Q^{\C,+}[\mS_N^\sigma \F, \mS_N^\sigma \F](r),\\
  Q^{\sigma,-}[\F,\F](r) &= \sum_{l,m\in K}\hat{B}_N(m,m)\hat{\F}_l \sigma_N^2(m) \hat{\F}_m 
  E_{l+m}(r) = Q^{\C,-}[\F, \mS_N^\sigma\mS_N^\sigma \F](r),
\end{align}
for any $r\in\X$. Following the boundedness of the truncated collision operator proven in
\cite{pareschi2000Spectral}, one concludes that
\[
\|Q^\sigma[\F,\F]-Q^\C[\F,\F]\|_2\leq\|Q^{\sigma,+}[F,F]-Q^{\C,+}[\F,\F]\|_2+\|Q^{\sigma,-}[\F,\F]-Q^{\C,-}[\F,\F]\|_2\leq
O(N^{-2}).
\]
Hence, EFM has second-order accuracy in approximating the truncated collision operator, if the
distribution is smooth enough (due to the smoothing effect of the Boltzmann collision operator
\cite{Barbaroux2017}, we only the initial value is smooth enough). This order of convergence will be
also numerically verified in the next section.

\subsection{Proof of Theorem \ref{thm:EFM}}\label{sec:proof}
This subsection provides the proof of Theorem \ref{thm:EFM}.
\begin{proof}[Proof of Theorem \ref{thm:EFM}]
  The argument in Section \ref{sec:SymFGM2} indicates that if $\f(t=0,\bv)\geq0$, then the
  coefficients $A_{pq}^{rs}$ and the initial values $\F^0_r$ satisfy all the three conditions in
  Condition \ref{cond:DVM}.
  
  The symmetry relation \eqref{eq:sym_B} of $\hat{B}(l,m)$ and the definition of
  $\hat{B}_N^\sigma(l,m)$ indicates $\bone_N(l+m)\left(
  \hat{B}_N^\sigma(l,m)-\hat{B}_N^\sigma(m,m)\right)=0$, i.e. $Q^\sigma_{0}=0$.  Noticing that the zero
  frequency $\hat{\F}_0(t)=\frac{1}{N^d}\sum_{r\in X}\F_r(t)$, one can directly obtain the
  conservation of mass \eqref{eq:conservationMass}.

  Since $\f^0(v)\geq 0$, $\F^0_r \geq 0$ by construction. If there exists $t'>0$ and $r\in\X$ such
  that $\F_r(t')=0$ and $\F_p(t')\geq0$ for any other $p\in\X$, then
  \begin{equation*}
    \od{F_r(t)}{t}\mid_{t=t'}=\sum_{p,q,s\in X}(A^\sigma)_{pq}^{rs}\F_p\F_q
    \geq0,
  \end{equation*}
  which indicates $\F_r(t)\geq0$ for all $t>0$ and $r\in\X$.
    
  The symmetry relation and non-negativity of $(A^\sigma)_{pq}^{rs}$ indicate the discrete H-theorem
  \eqref{eq:Htheorem}
  \begin{equation*}
    \sum_{r\in X}Q^\sigma_r \ln(\F_r) = \frac{1}{4} \sum_{p,q,r,s\in X} (A^\sigma)_{pq}^{rs}\ln\left(
    \frac{\F_r\F_s}{\F_p\F_q} \right) [\F_p\F_q-\F_r\F_s]\leq0.
  \end{equation*}
  This completes the proof.
\end{proof}

% \input{numerical}
% vim: tw=70:spell
\section{Numerical tests}\label{sec:numerical}
This section describes several numerical tests to demonstrate the properties of EFM and to compare
with the Fourier Galerkin method (FGM) in \cite{pareschi2000Spectral} and the positivity preserving
spectral method (PPSM) in \cite{pareschi2000stability}.

\subsection{Implementation}\label{sec:implementation}
It is pointed out in Section \ref{sec:result} that the fast algorithms in \cite{Mouhot, Hu2016} can
be applied to EFM without affecting the H-theorem. To show this, one needs to check that the fast
algorithms do not violate the first two conditions in Condition \ref{cond:DVM}.

These fast algorithms are based on an approximation of $\hat{B}_N(l,m)$
\eqref{eq:low_rank_approx_filtered} of the following form.
\begin{equation}\label{eq:fastAlg}
  \hat{B}_{N}(l,m) \approx
  \hat{B}_{N,\fast}(l,m):=\sum_{t=1}^M \alpha_{l+m}^t \beta_{l}^t \gamma_{m}^t.
\end{equation}
After the filtering is applied, one obtains a similar approximation for $\hat{B}_N^{\sigma}(l,m)$
\begin{equation}
  \hat{B}^{\sigma}_N(l,m)\approx 
  \hat{B}^{\sigma}_{N,\fast}(l,m) :=\sum_{t=1}^M \alpha_{l+m}^t \left(\sigma_N(l)\beta_l^t\right)
  \left(\sigma_N(m)\gamma_m^t\right).
\end{equation}
It can be verified that both kernels satisfy the symmetry relation
\begin{eqnarray}\label{eq:symmetryBfast}
  &\hat{B}_{N,\fast}(l,m)= \hat{B}_{N,\fast}(m,l)= \hat{B}_{N,\fast}(-l,m), \\
  &\hat{B}_{N,\fast}^{\sigma}(l,m)= \hat{B}_{N,\fast}^{\sigma}(m,l)= \hat{B}_{N,\fast}^{\sigma}(-l,m),
\end{eqnarray}
which indicates Condition \ref{cond:DVM}.\ref{it:symmetry} is valid for the fast algorithms.

To see that the fast algorithms do not affect the non-negativity of $G^\sigma(\y,\z)$, we use
the fast algorithm in \cite{Mouhot} with $d=2$ and $\tilde{B}=1$ as an example. The first step of
this algorithm writes $\y$ and $\z$ in \eqref{eq:def_hatB} in the polar coordinates $\y=\rho\be$
and $\z=\rho_*\be_*$:
\begin{equation}\label{eq:spherical_hatB}
  \hat{B}_N(l,m)=\frac{1}{4}
    \int_{\bbS^1}\int_{\bbS^1}\delta(\be\cdot\be_*)
    \left[ \int_{-R}^RE_{l}(\rho\be)\dd\rho \right]
    \left[ \int_{-R}^RE_{m}(\rho_*\be_*)\dd\rho_* \right]
    \dd\be\dd\be_*.
\end{equation}
Let $\psi_R(l,\be)=\int_{-R}^RE_{l}(\rho\be)\dd\rho$, then
\[
\hat{B}_N(l,m)=\frac{1}{4}\int_{\bbS^1}\int_{\bbS^1}\delta(\be\cdot\be_*)
\psi_R(l,\be)\psi_R(m,\be_*)\dd \be\dd\be_*.
\]
Integrating it with respect to $\be_*$ yields
\begin{equation}\label{eq:fast_hatB}
  \hat{B}_N(l,m) =\int_0^{\pi}\psi_R(l,\be_\theta)\psi_R(m,\be_{\theta+\pi/2})
  \dd\theta.
\end{equation}
Substituting \eqref{eq:spherical_hatB} into \eqref{eq:tildeG} gives rise to
\begin{equation}\label{eq:spherical_tildeG}
  G^\sigma(\y,\z)= \frac{1}{4}
  \int_{\bbS^1}\int_{\bbS^1}\delta(\be\cdot\be_*)
  \left[ \int_{-R}^R\chi_N^\sigma(\rho\be  -\y)\dd\rho \right]
  \left[ \int_{-R}^R\chi_N^\sigma(\rho\be_*-\z)\dd\rho_* \right]
  \dd\be\dd\be_*.
\end{equation}
Let $\phi^\sigma_R(\y,\be)=\int_{-R}^R\chi_N^\sigma(\rho\be-\y)\dd\rho$.  Apparently,
$\phi^\sigma_R(\y,\be)\geq0$ due to $\chi_N^\sigma(\y)\geq0$ for any $\y\in\bbR^2$.  Then integrating
\eqref{eq:spherical_tildeG} with respect to $\be_*$ yields
\begin{equation}
  G^\sigma(\y,\z)
  =\int_0^{\pi}\phi^\sigma_R(\y,\be_\theta)\phi^\sigma_R(\z,\be_{\theta+\pi/2})
  \dd\theta.
\end{equation}

The idea of the fast algorithm is to replace the integration in
\eqref{eq:fast_hatB} with a quadrature formula. More precisely
\eqref{eq:fast_hatB} is approximated by
\begin{equation}\label{eq:discrete_hatB}
    \hat{B}_{N,\fast}(l,m) =
    \sum_{t=1}^M\frac{\pi}{M}\psi_R(l,\be_{\theta_t})\psi_R(m,\be_{\theta_t+\pi/2}).
\end{equation}
Similarly to \eqref{eq:discrete_hatB}, one obtains
\begin{equation}
  G^\sigma_\fast(\y,\z) = 
  \sum_{t=1}^M\frac{\pi}{M}\phi^\sigma_R(\y,\be_{\theta_t})\phi^\sigma_R(\z,\be_{\theta_t+\pi/2}).
\end{equation}
Since $\phi^\sigma_R(\y,\be)\geq0$ for any $\y\in\bbR^2$, $\be\in\bbS^1$,
$G^\sigma_\fast(\y,\z)\geq0$ for any $\y,\z\in\bbR^2$.  Hence, the fast algorithm does not
destroy the non-negativity of $G^\sigma_\fast(\y,\z)$.

As we pointed out in Section \ref{sec:sym_FCM}, an aliased convolution can be directly used to
calculate \eqref{eq:EFM}. Since the accuracy of EFM is only second order, the smoothing filter is
the main source of the error.  In the fast algorithms, the number $M$ in \eqref{eq:fastAlg} perhaps
can be smaller than that in \cite{Mouhot, Hu2016}.

In the above discussion, we only study the case when $N$ is odd. The case of even $N$ values can be reduced to
the odd $(N-1)$ case by setting the coefficient of a mode $k$ to be zero $0$ if any component of
$k=(k_1,\ldots,k_d)$ is equal to $-N/2$.
%and the point values $\cdot_p$ as $0$ if any component of $p$ is equal to $-Nh / 2$.

For the time discretization, the third-order strong stability-preserving Runge-Kutta method proposed
in \cite{gottlieb2001strong} is employed in the discretization of time.  In all tests, the time step
is chosen as $\Delta t = 0.01$.

\subsection{Numerical results}\label{sec:numericalresult}
The test problems used here are solutions of the space-homogeneous Boltzmann equation for Maxwell
molecules ($\mB(\bg,\omega)=\frac{1}{2\pi}$ in 2D and $\mB(\bg,\omega)=\frac{1}{4\pi}$ in 3D).

\paragraph{Example 1 ($2D$ BKW solution).}
The first example is the well-known $2D$ BKW solution, obtained independently in
\cite{bobylev1975exact} and \cite{krook1977exact}.  The exact solution takes the form
\begin{equation}
  \f(t, \bv)=\frac{1}{2\pi S}\exp\left( -\frac{|\bv|^2}{2S} \right)
  \left( \frac{2S-1}{S}+\frac{1-S}{2S^2}|\bv|^2 \right),
\end{equation}
where $S=1-\exp(-t/8)/2$. The BKW solution allows one to check the accuracy, positivity of the
solution, and the entropy of the proposed method. Here we set the truncation radius $R=6$ in the
tests.

\begin{figure}[ht]
    \subfigure[$N=16$]{%
        \begin{overpic}[width=.48\textwidth,clip]{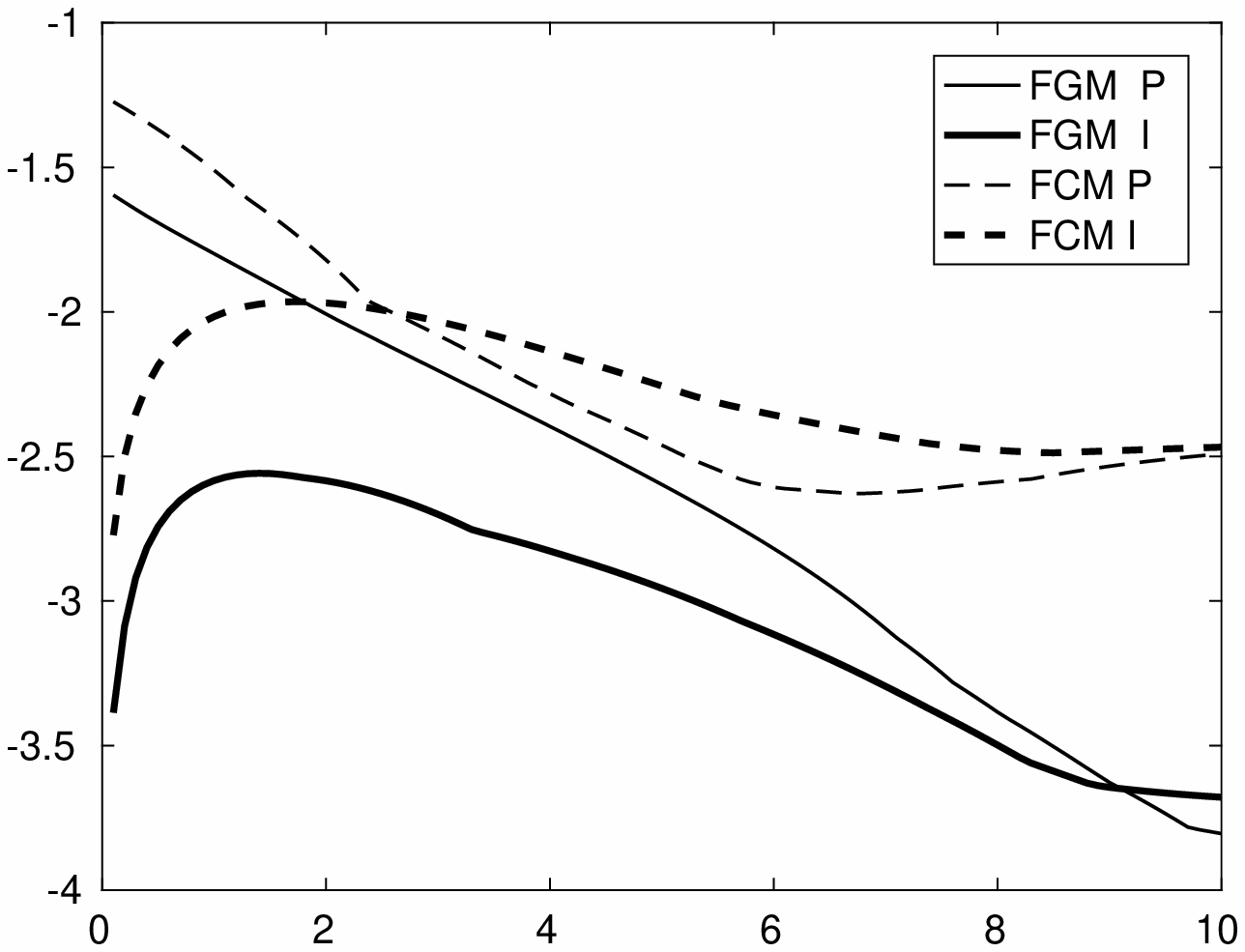}
            \put(51,0){$t$}
            \put(-4,30){\small\rotatebox{90}{$\log_{10}(\epsilon)$}}
        \end{overpic}
    }\qquad
    \subfigure[$N=32$]{%
        \begin{overpic}[width=.48\textwidth,clip]{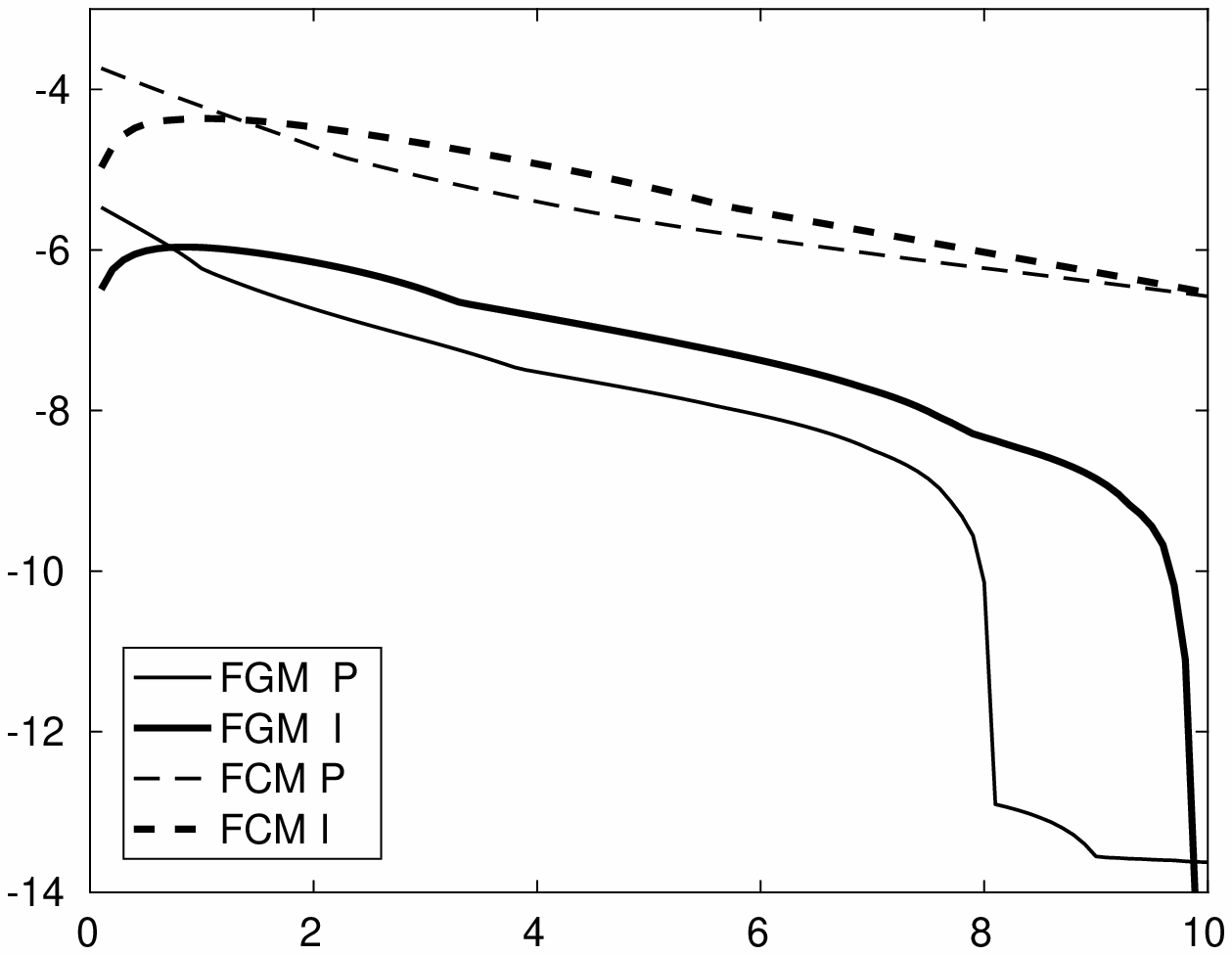}
            \put(51,0){$t$}
            \put(-4,30){\small\rotatebox{90}{$\log_{10}(\epsilon)$}}
        \end{overpic}
    }
    \caption{\label{fig:positiveError} Positivity error of FGM and FCM with the initial value
      prepared by orthogonal projection (P) and interpolation (I) in the $\log_{10}$ scale. Since
      the positivity error of EFM is strictly zero, its result is not plotted in the figure.}
\end{figure}

Both FGM \eqref{eq:FG} or FCM \eqref{eq:FCM} result in good approximations of the exact solution
at the collocation points. However, the solutions are not non-negative. Even if we use interpolation
rather than orthogonal projection to prepare the initial values, the solutions of these methods
still fail to be non-negative. The positivity of the solution is measured by the positivity error
defined via
\begin{equation}
  \epsilon := \frac{\sum_{q\in\X}|\F_q|-\sum_{q\in\X}\F_q}{\sum_{q\in\X}|\F_q|}.
\end{equation}
Figure \ref{fig:positiveError} shows that both FGM and FCM fail to preserve the positivity of the
solution at the collocation points no matter whether the initial value is given by orthogonal
projection or interpolation. On the other hand, the positivity error of EFM
is strictly zero, thanks to the modification in \eqref{eq:tildeG}.

Table \ref{tab:accuracyBKW} summarizes the $\ell_1$, $\ell_2$ and $\ell_{\infty}$ errors of
\eqref{eq:modifiedIVP} at time $t=0.01$. Here the $\ell_\mathrm{p}$ relative errors for
$\mathrm{p}=1,2,\infty$ are defined by
\begin{equation}
    \frac{\|\F-\f\|_\mathrm{p}}{\|\f\|_\mathrm{p}} = \frac{\left(\sum_{q\in\X}
    |\F_q-\f(q)|^\mathrm{p}\right)^{1/\mathrm{p}} }
    {\left(\sum_{q\in\X} |\f(q)|^\mathrm{p}\right)^{1/\mathrm{p}}},
\end{equation}
where $\f(q)$ is the exact solution at $q\in\X$. The numerical results also show that the
convergence rate of EFM is of second order.

\begin{table}[!ht]
  \centering
  \begin{tabular}{|c|c|c|c|c|c|c|}
    \hline
    $N$ & $\ell_1$ error  & rate & $\ell_2$ error  & rate & $\ell_{\infty}$ error & rate \\ \hline
    16  & $4.68\times10^{-3}$ &     & $3.23\times10^{-3}$ &     & $3.12\times10^{-3}$&     \\ \hline 
    32  & $1.72\times10^{-3}$ & 1.44& $1.36\times10^{-3}$ & 1.25& $1.40\times10^{-3}$& 1.15\\ \hline
    64  & $5.54\times10^{-4}$ & 1.64& $4.56\times10^{-4}$ & 1.58& $5.57\times10^{-4}$& 1.34\\ \hline
    128 & $1.55\times10^{-4}$ & 1.84& $1.29\times10^{-4}$ & 1.82& $1.73\times10^{-4}$& 1.68\\ \hline
    256 & $4.05\times10^{-5}$ & 1.93& $3.42\times10^{-5}$ & 1.92& $4.73\times10^{-5}$& 1.87\\ \hline
    512 & $1.03\times10^{-5}$ & 1.97& $8.76\times10^{-6}$ & 1.96& $1.22\times10^{-5}$& 1.94\\ \hline
  \end{tabular}
  \caption{\label{tab:accuracyBKW}The $\ell_1$, $\ell_2$ and $\ell_\infty$ errors and convergence
    rates for the BKW solution at time $t=0.01$ with $R=6$. }
\end{table}

As discussed in Section \ref{sec:implementation}, the fast algorithm in \cite{Mouhot} can be applied
to EFM to accelerate the computation. In \eqref{eq:fast_hatB}, the integration on $[0,\pi)$ can be
  reduced to $[0,\pi/2)$ and $M$ is equal to the number of samples within $[0,\pi/2)$. Table
      \ref{tab:fastError} presents the $\ell_1$ error for multiple values of $M$, noticing that
      $M=2$ is good enough in practice while in \cite{Mouhot} the authors suggest $M\geq4$.
\begin{table}[!ht]
  \centering
  \begin{tabular}{|c|c|c|c|}
    \hline
    $N$ & $M=2$ & $M=3$ & $M=32$ \\\hline
    16  & $4.6852\times10^{-3}$ &   $4.6826\times10^{-3}$ &   $4.6830\times10^{-3}$ \\ \hline   
    32  & $1.7241\times10^{-3}$ &   $1.7244\times10^{-3}$ &   $1.7245\times10^{-3}$ \\ \hline
    64  & $5.5368\times10^{-4}$ &   $5.5388\times10^{-4}$ &   $5.5394\times10^{-4}$ \\ \hline
    128 & $1.5485\times10^{-4}$ &   $1.5488\times10^{-4}$ &   $1.5489\times10^{-4}$ \\ \hline
    256 & $4.0513\times10^{-5}$ &   $4.0516\times10^{-5}$ &   $4.0517\times10^{-5}$ \\ \hline
  \end{tabular}
  \caption{\label{tab:fastError}The $\ell_1$ error of EFM with fast algorithm in \cite{Mouhot} for
    multiple choices of $N$ and $M$.}
\end{table}

\begin{figure}[!ht]
  \centering
  \includegraphics[width=0.5\textwidth,clip]{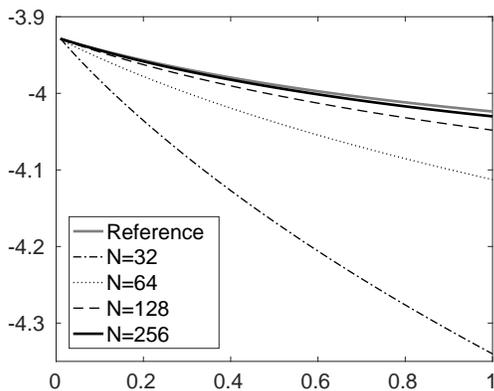}
  \caption{\label{fig:Entropy}The evolution of the entropy of EFM for multiple $N$ values.}
\end{figure}

In order to demonstrate that the proposed method satisfies the H-theorem numerically, we define a
time-dependent discrete entropy function
\begin{equation}
  \eta(t)=\left(\frac{2T}{N}\right)^d\sum_{q\in\X}\F_q(t) \ln\F_q(t).
\end{equation}
The evolution of the entropy, plotted in Figure \ref{fig:Entropy}, shows that as the number of the
discrete points $N$ increases the discrete entropy converges to the one of the exact solution.

\begin{figure}[!ht]
  \centering
  \includegraphics[width=0.5\textwidth,clip]{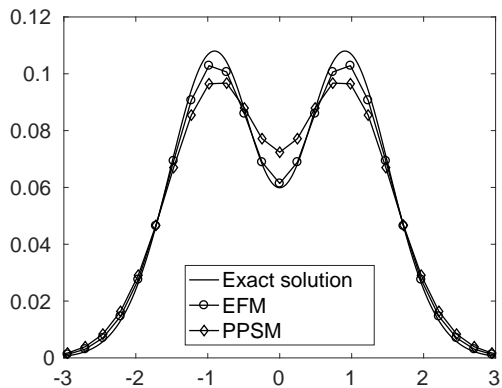}
  \caption{\label{fig:PositiveVSHtheorem} Comparison between PPSM and EFM at time $t=1$ with $N=64$
    for BKW solution.}
\end{figure}

As a comparison with PPSM, Figure \ref{fig:PositiveVSHtheorem} presents the numerical solutions in
the $v_1$ direction of PPSM and EFM at $t=1$ with $N=32$. The smoothing filter used for EFM results
in much less dissipation, thus leading to a much better agreement with the exact solution. Finally,
Figure \ref{fig:BKWN} shows that as $N$ increases the solution of EFM converges rapidly to the exact
solution.

\begin{figure}[!ht]
    \centering
    \subfigure{%
      \includegraphics[width=0.48\textwidth,clip]{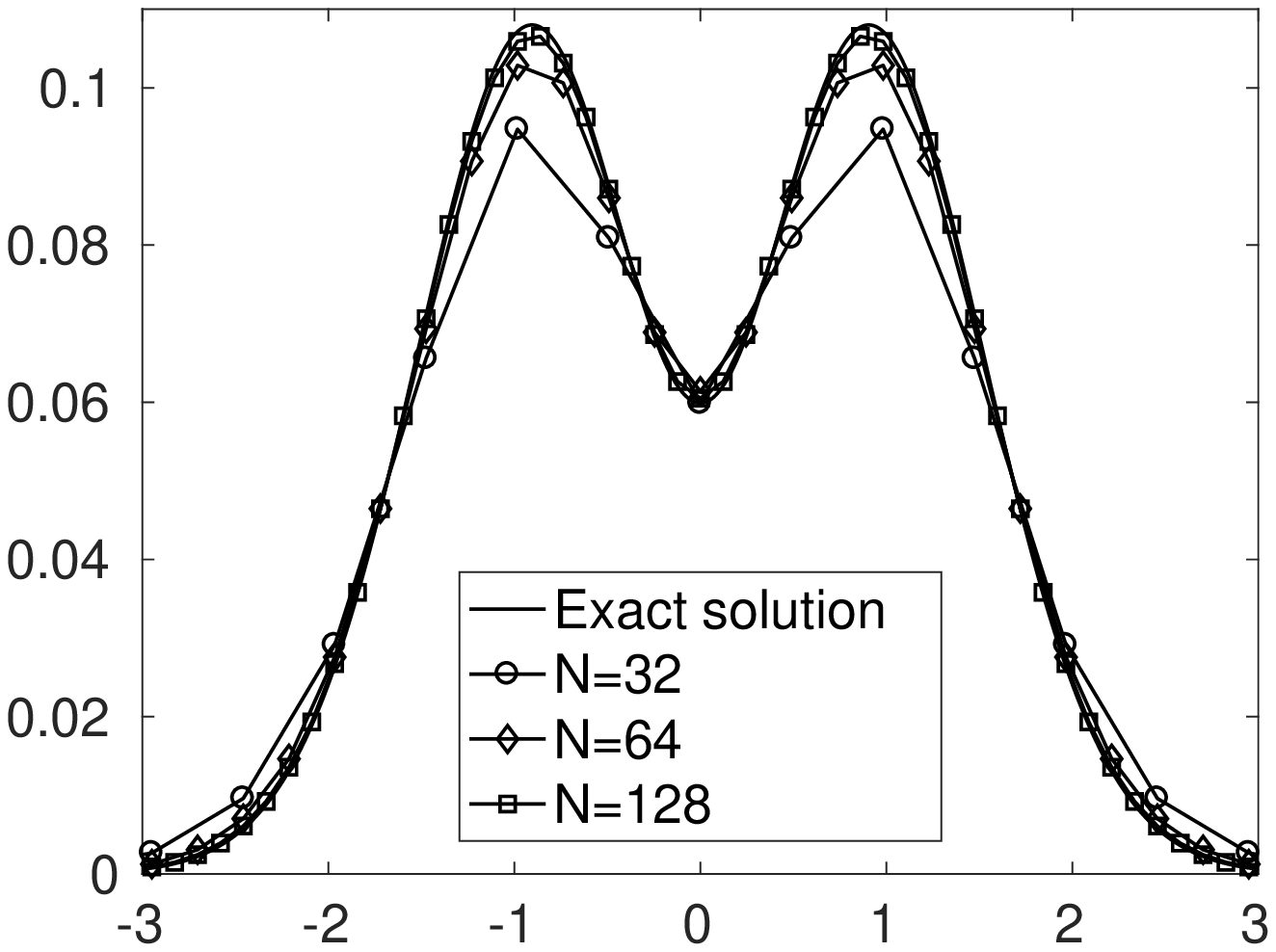}\qquad
    }
    \subfigure{%
      \includegraphics[width=0.45\textwidth,clip]{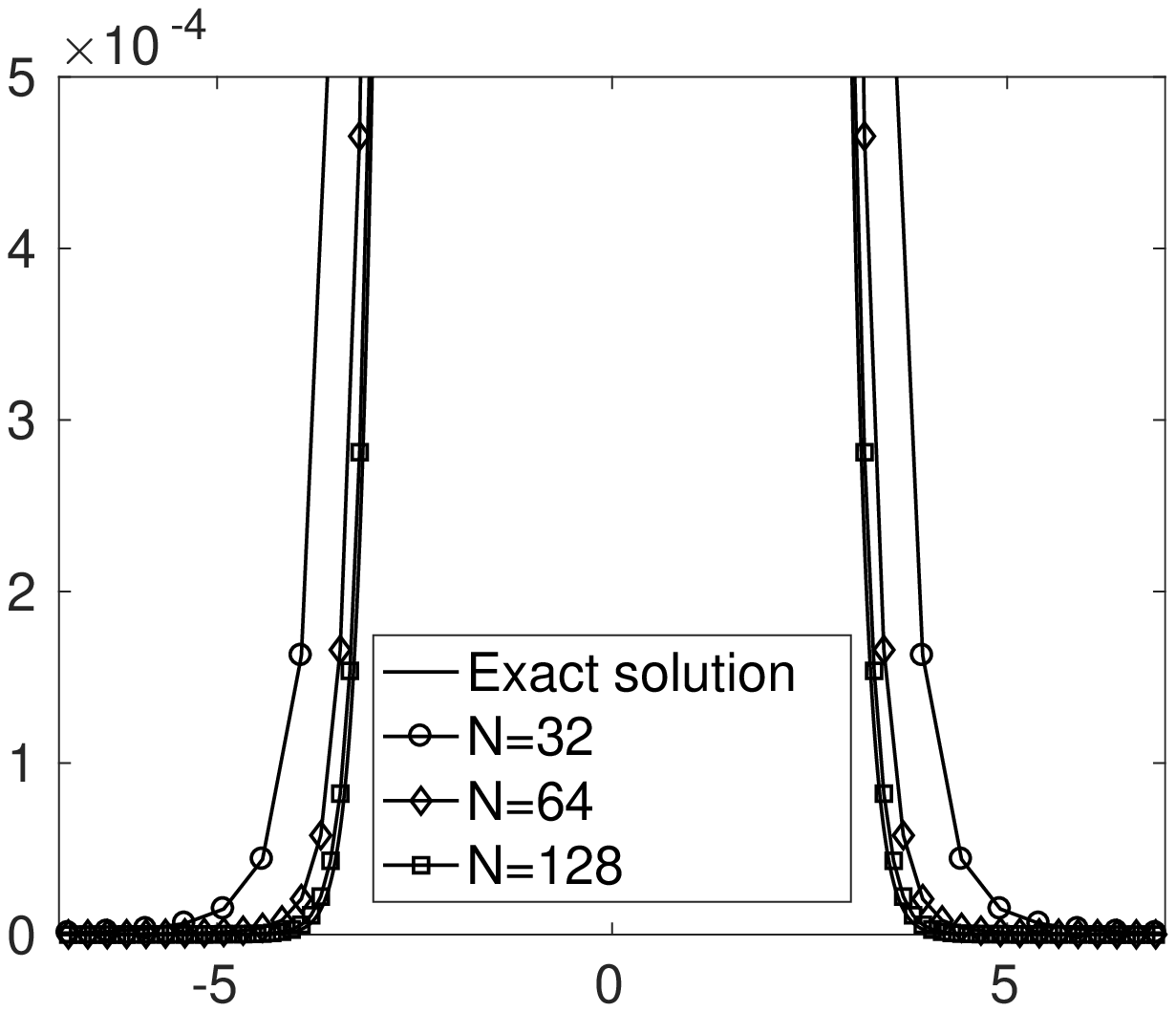}
    }
    \caption{\label{fig:BKWN}Numerical solution of EFM for multiple $N$ values with the BKW solution
      at time $t=1$ on different scales.}
\end{figure}

\paragraph{Example 2 (Bi-Gaussian initial value).}  
Another frequent example is a problem with the bi-Gaussian initial value
\begin{equation}
  \f(t=0,\bv) = \frac{1}{4\pi}\left( 
  \exp\left( -\frac{|\bv-\bu_1|^2}{2} \right) 
  +\exp\left( -\frac{|\bv-\bu_2|^2}{2} \right) 
  \right),\quad 
\end{equation}
where $\bu_1=(-2,0)^\T$ and $\bu_2=(2,0)^\T$. This is solved for the Maxwell molecules (2D in
velocity) with radius $R=6$. Figure \ref{fig:PositiveVSHtheoremBiG} shows the numerical results of
PPSM and EFM. The reference solution is calculated by the Fourier spectral method with $N=400$ and
$R=8$. It is clear that EFM solution is much closer to the reference solution. Figure
\ref{fig:BiGN} demonstrates that as $N$ increases EFM converges rapidly to the reference solution.

\begin{figure}[ht]
  \centering
  \includegraphics[width=0.5\textwidth,clip]{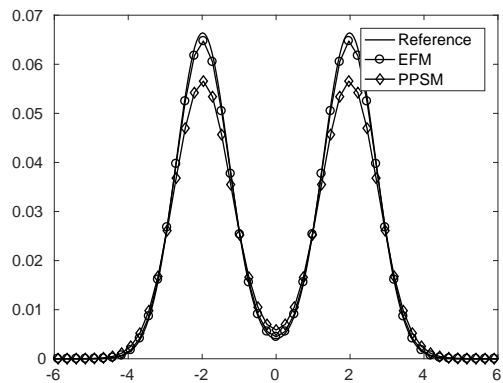}
  \caption{\label{fig:PositiveVSHtheoremBiG}Comparison between PPSM and EFM at $t=1$ with $N=64$ for
  the bi-Gaussian initial value.}
\end{figure}

\begin{figure}[!ht]
    \centering
    \subfigure{%
    \includegraphics[width=0.48\textwidth,clip]{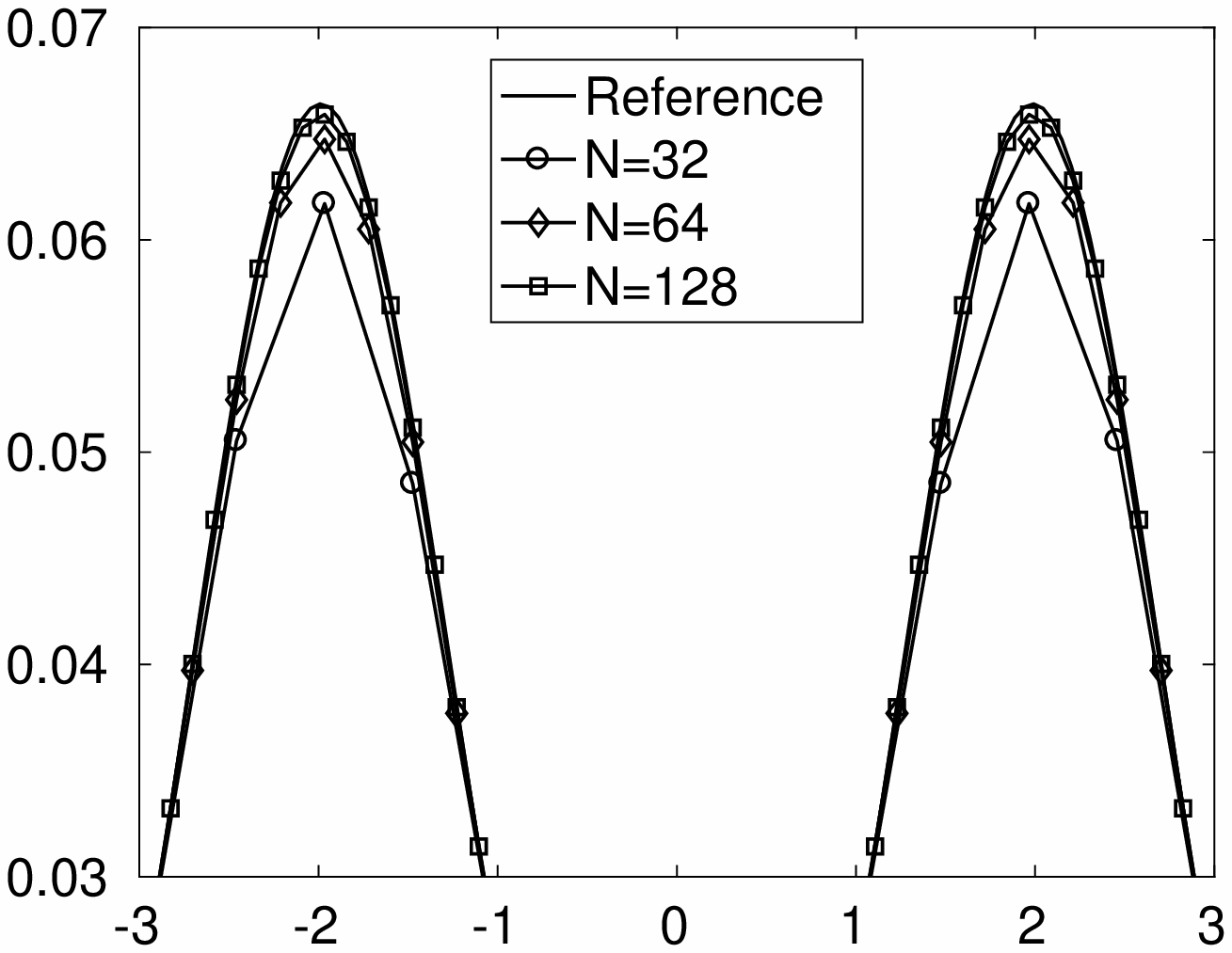}
    }\qquad
    \subfigure{%
    \includegraphics[width=0.45\textwidth,clip]{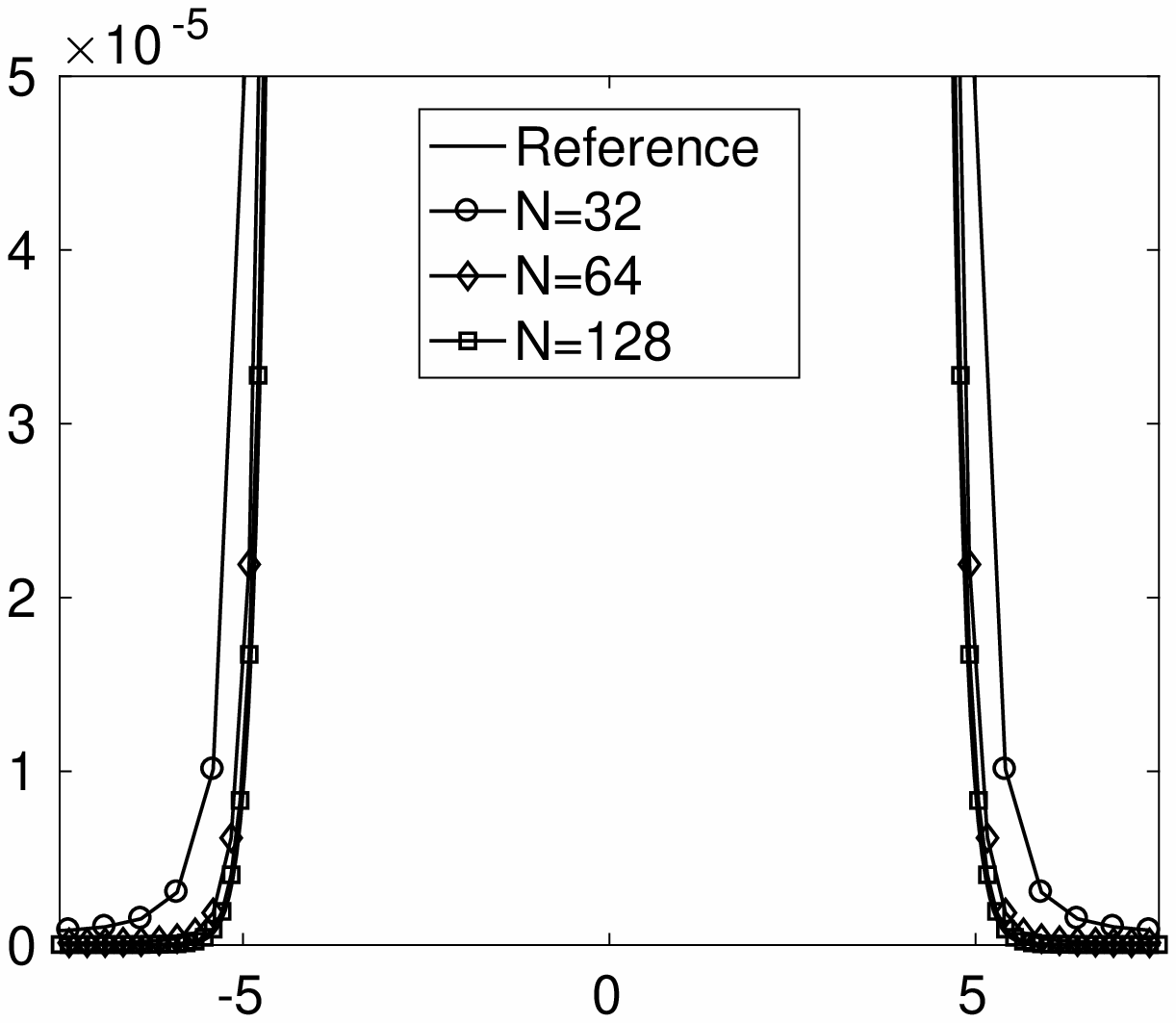}
    }
    \caption{\label{fig:BiGN}Numerical solution of EFM for multiple $N$ values with the bi-Gaussian
      initial value at time $t=1$ on different scales.}
\end{figure}

\paragraph{Example 3 (Discontinuous initial value).}
The initial condition given by
\begin{equation}\label{eq:discontinuous}
  \f(t=0,\bv)=\left\{ \begin{array}{ll}
    \frac{\rho_1}{2\pi T1}\exp\left( -\frac{|\bv|^2}{2T_1}
    \right),    & \text{ if }v_1 > 0,\\
    \frac{\rho_2}{2\pi T2}\exp\left( -\frac{|\bv|^2}{2T_2}
    \right),    & \text{ if } v_1 < 0,
  \end{array} \right.
\end{equation}
in this example is discontinuous. Here $\rho_1=\frac{6}{5}$ and the values of $\rho_2$, $T_1$ and
$T_2$ are uniquely determined by the three conditions
\[
\int_{\bbR^2}f(0,v)\dd\bv=\int_{\bbR^2}f(0,v)|\bv|^2/2\dd\bv=1, \quad 
\int_{\bbR^2}f(0,v)\bv\dd\bv=0.
\]

\begin{figure}[!ht]
    \subfigure[Profile of $\F(t=0.5,v_1,v_2=0)$]{
        \includegraphics[width=0.45\textwidth,clip]{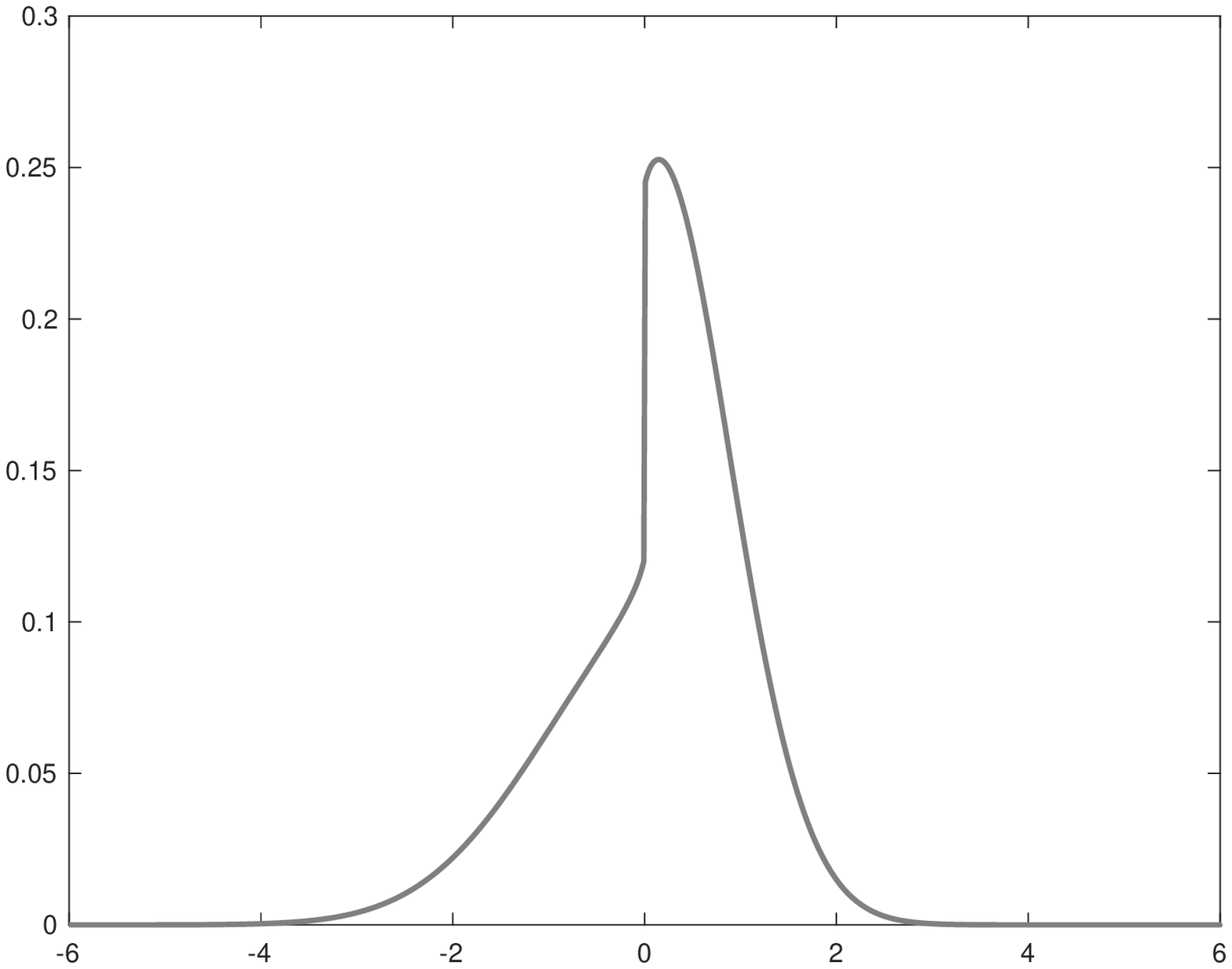}
    }
    \subfigure[Profile of $\F(t=0.5,\bv)$]{
        \includegraphics[width=0.5\textwidth,clip]{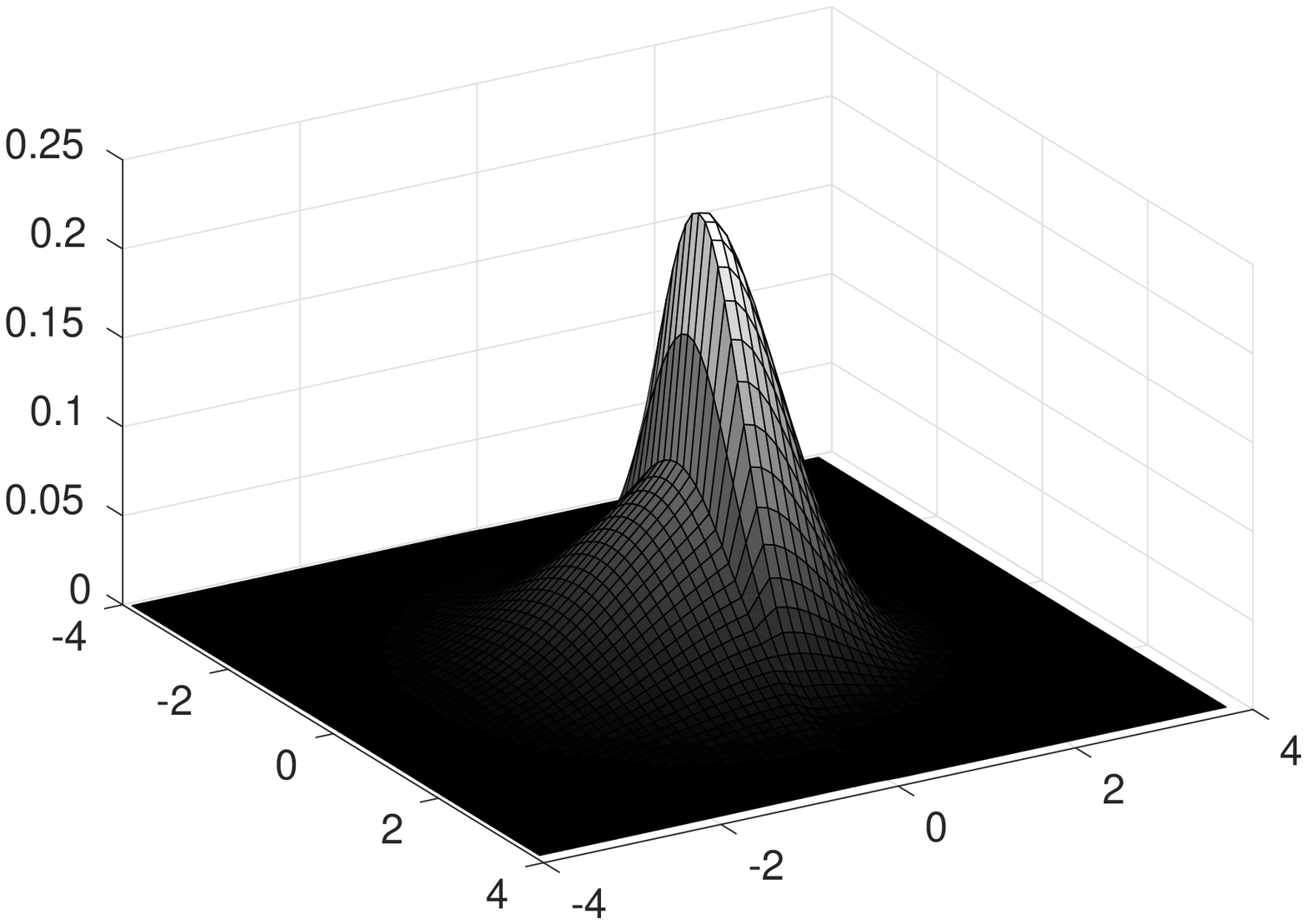}
    }
    \caption{\label{fig:Disc2D}Profile of $\F(t,\bv)$ with the discontinuous initial value
      \eqref{eq:discontinuous} at time $t=0.5$.}
\end{figure}
\newcommand\addImagesDIsc[2]{
    \subfigure[#1]{%
        \includegraphics[width=0.3\textwidth,clip]{#2.eps}
    }%
}
\begin{figure}[!ht]
    \centering
    \addImagesDIsc{$N=64$}{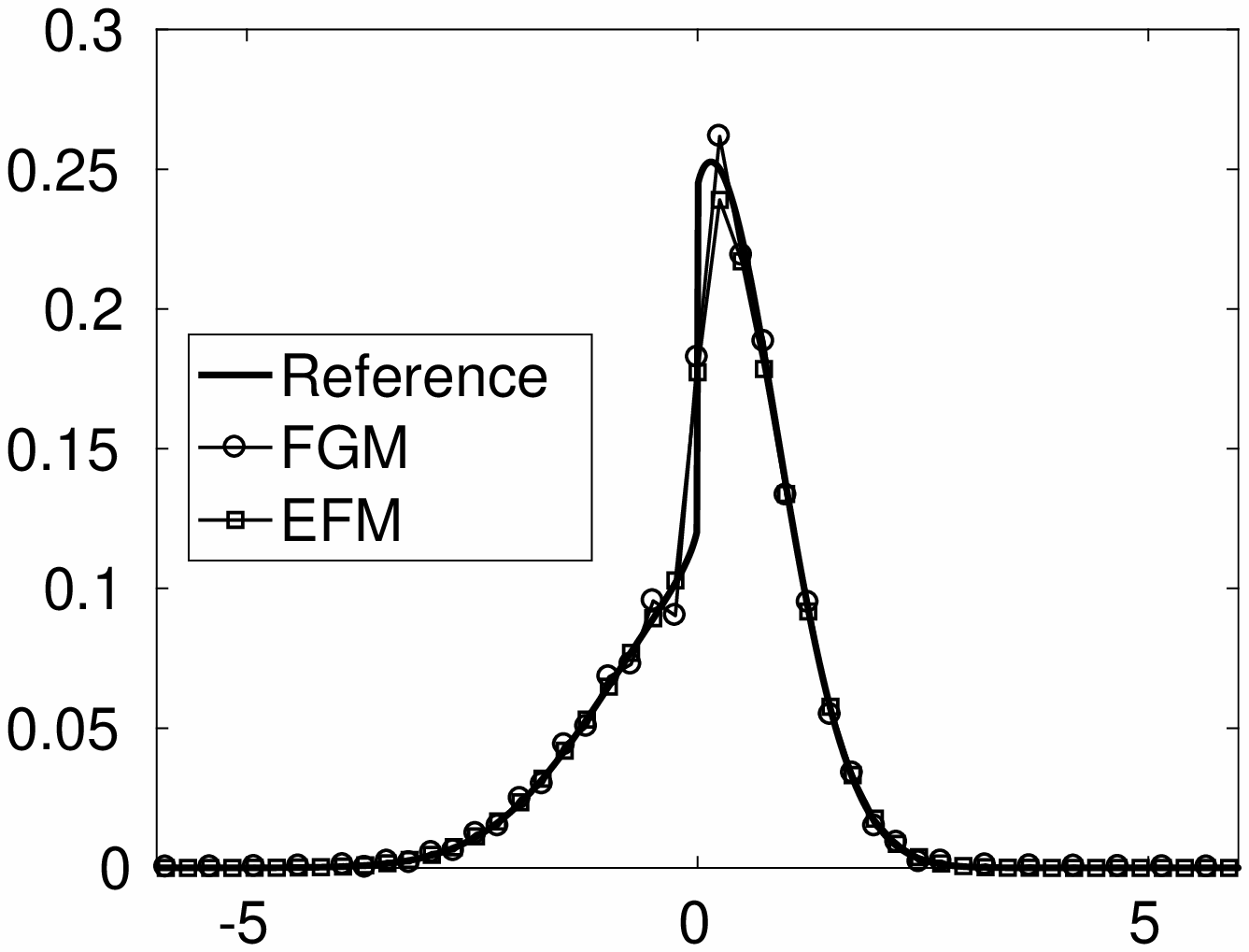}
    \addImagesDIsc{$N=64$}{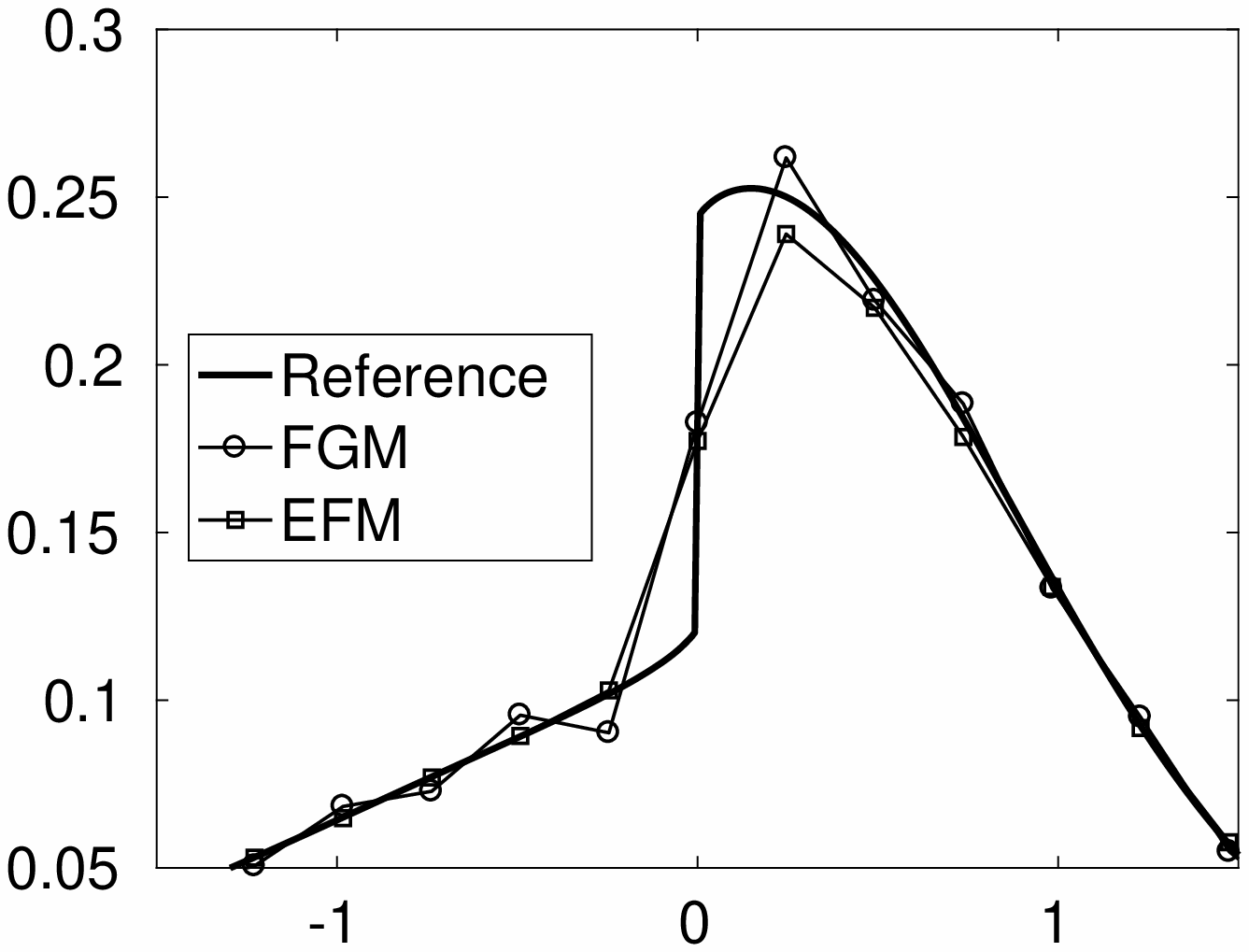}
    \addImagesDIsc{$N=64$}{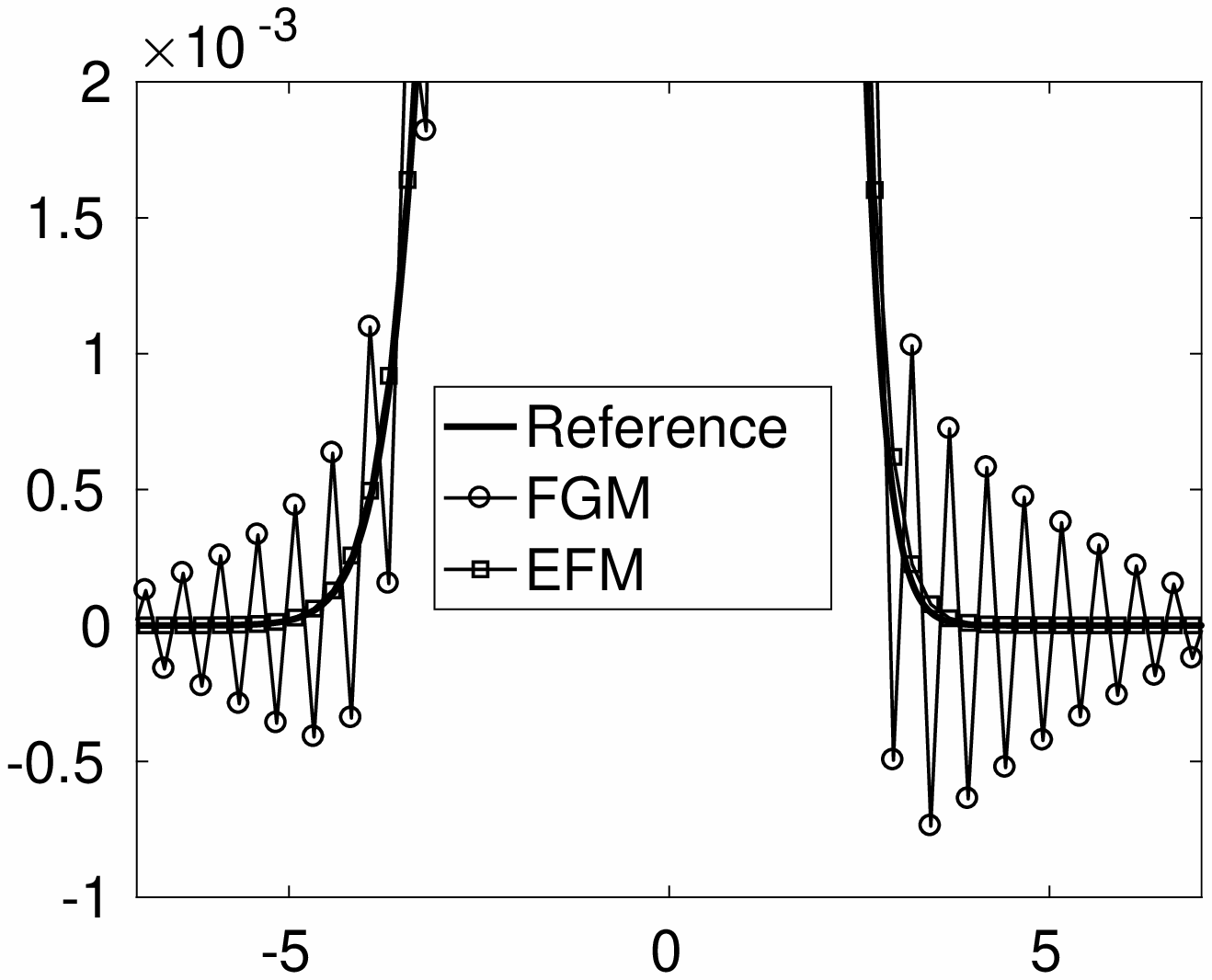}\\
    \addImagesDIsc{$N=128$}{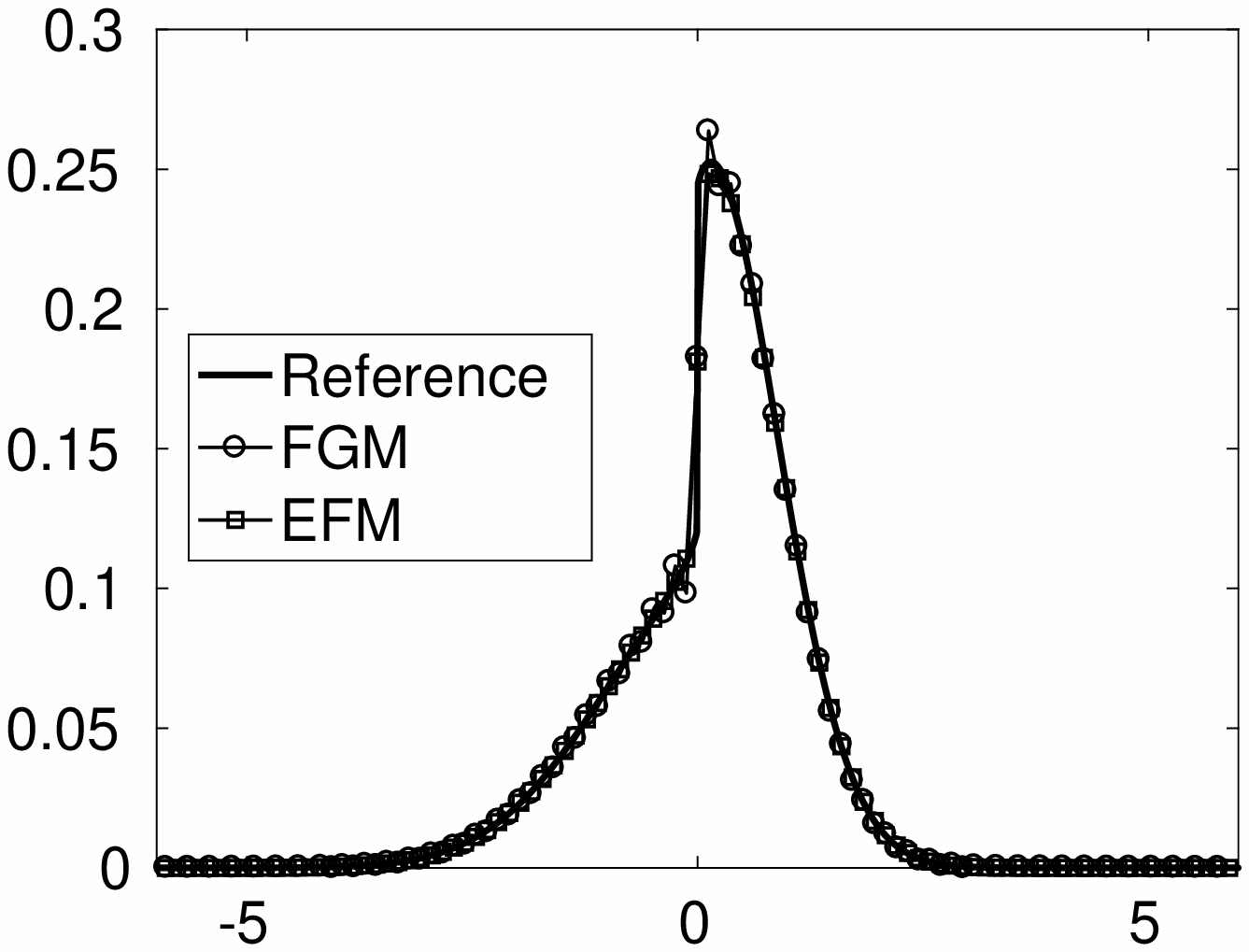}
    \addImagesDIsc{$N=128$}{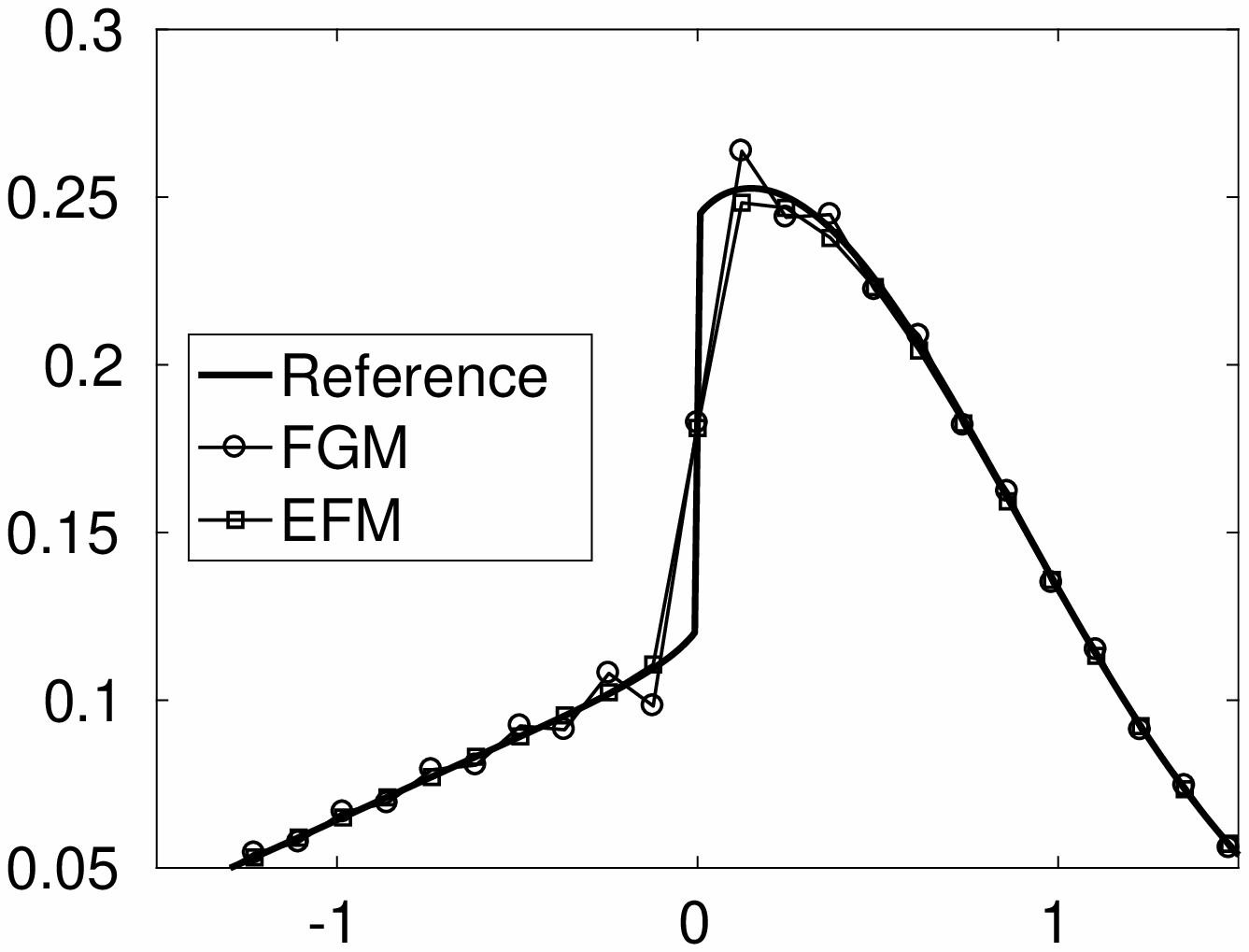}
    \addImagesDIsc{$N=128$}{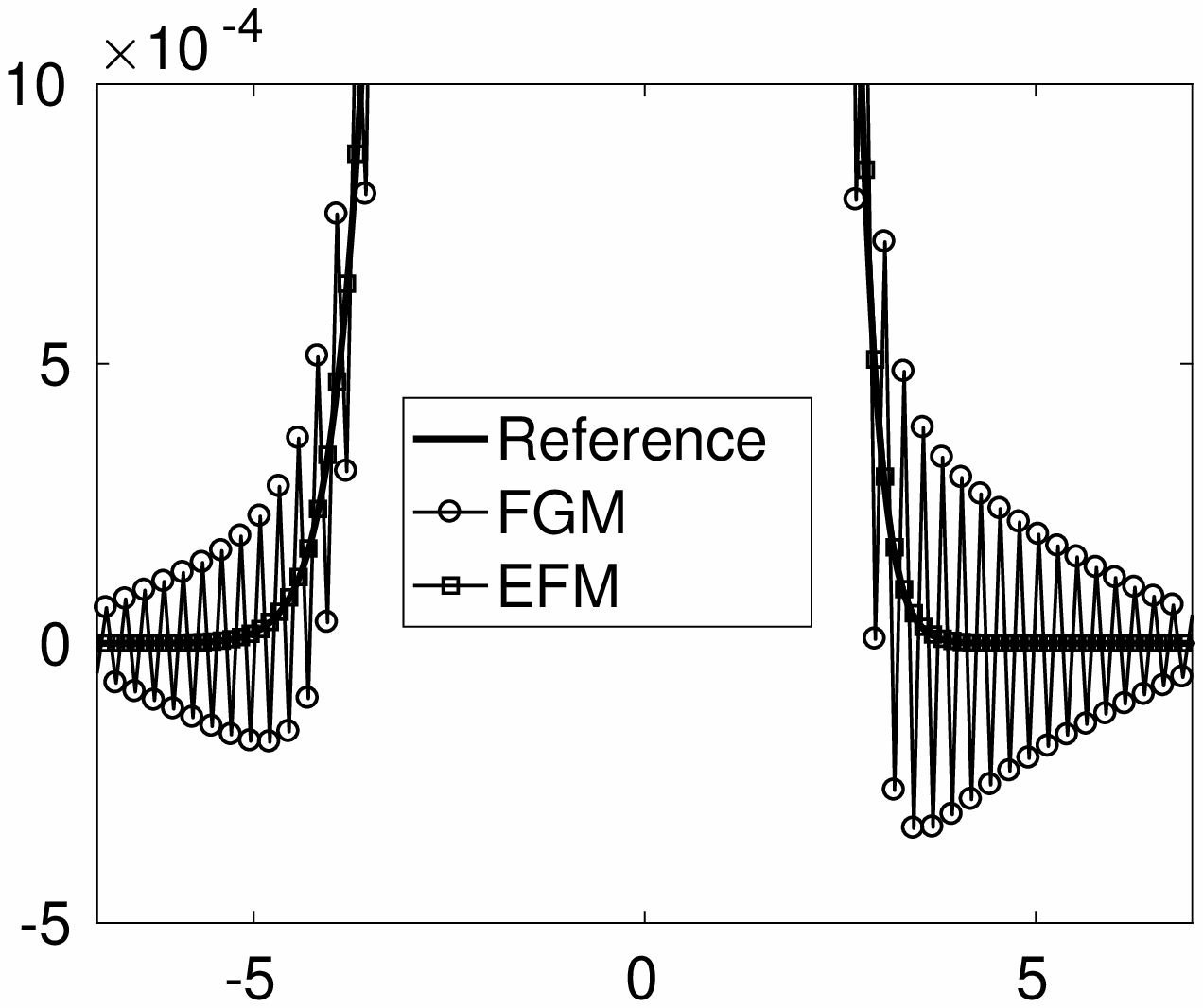}\\
    \addImagesDIsc{$N=256$}{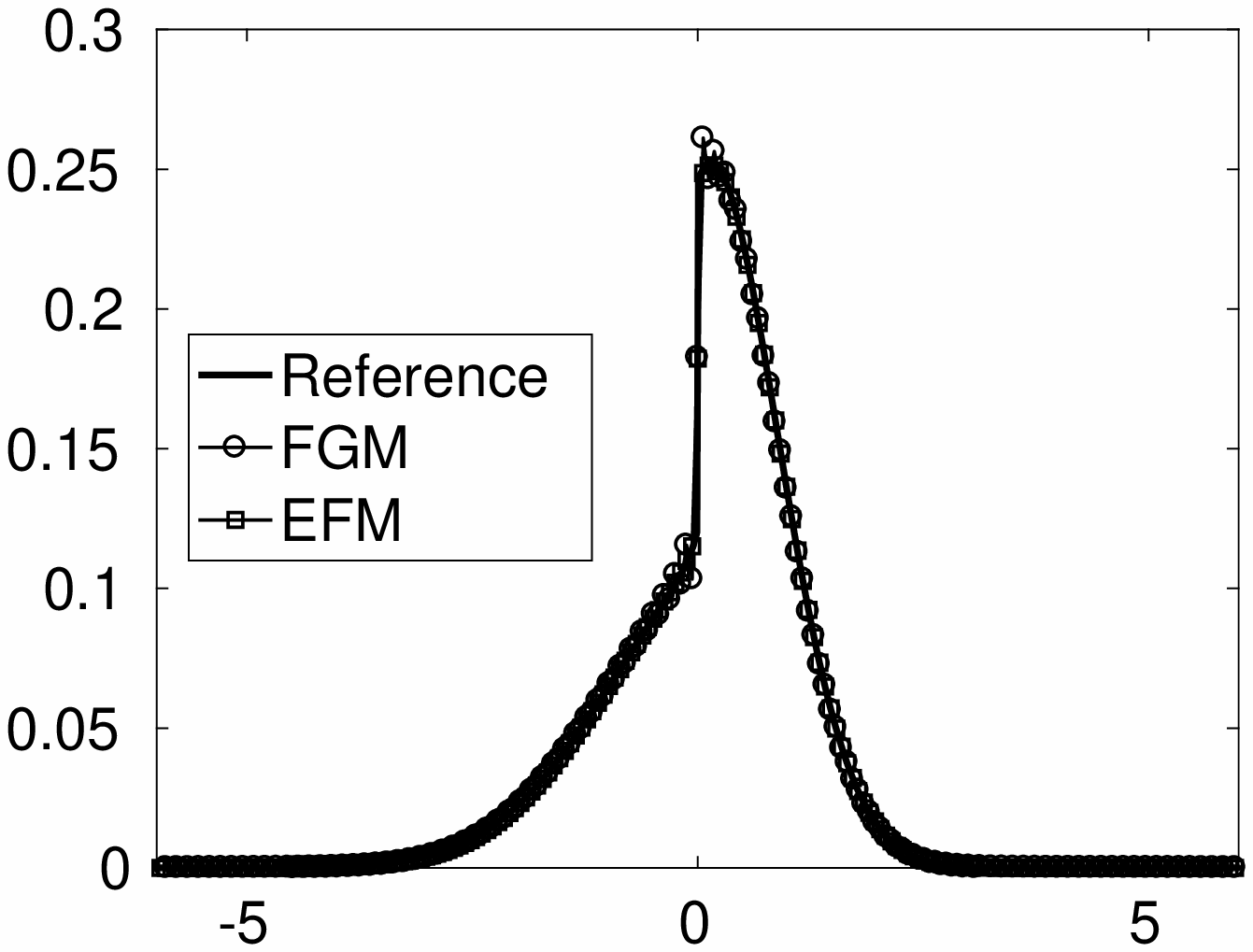}
    \addImagesDIsc{$N=256$}{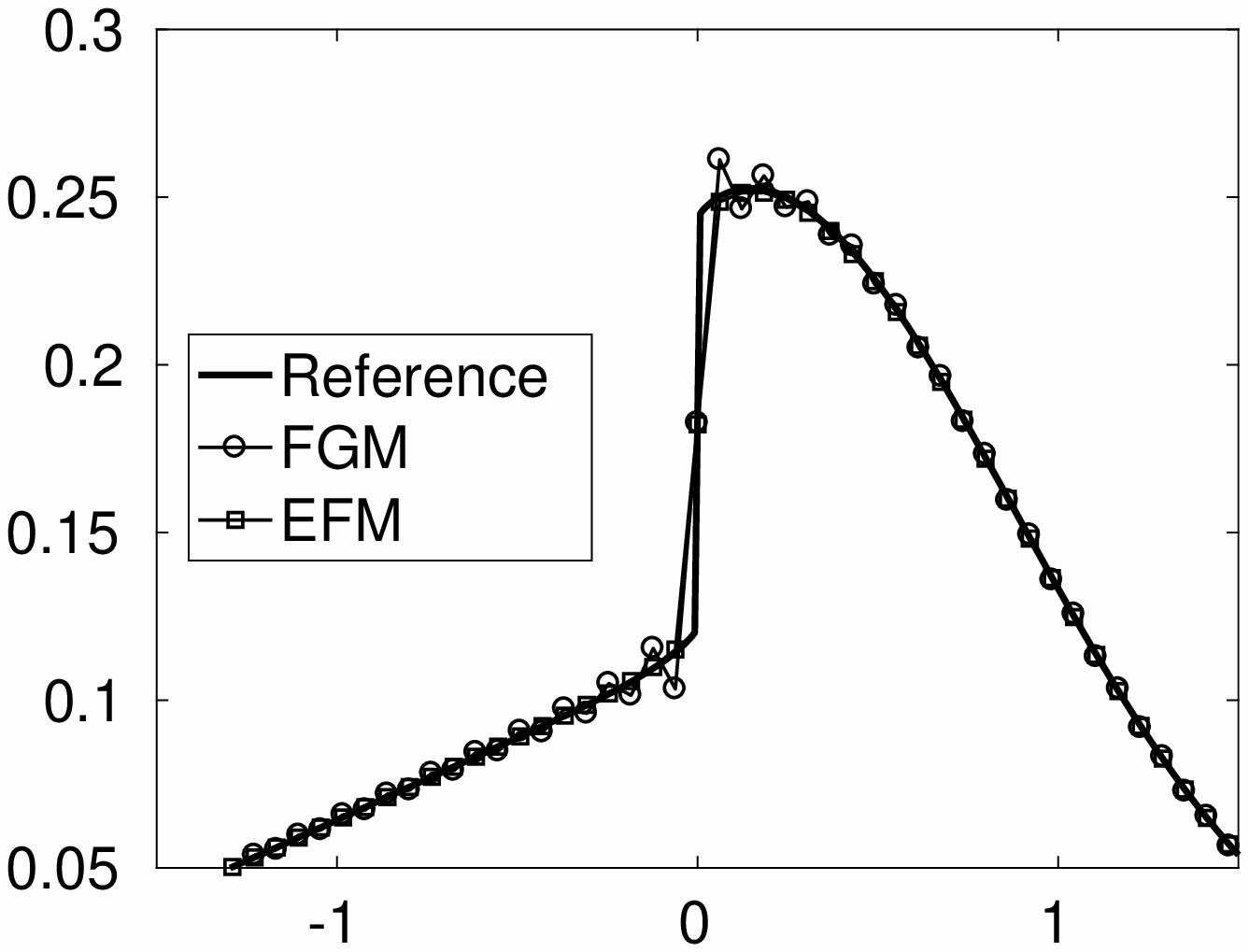}
    \addImagesDIsc{$N=256$}{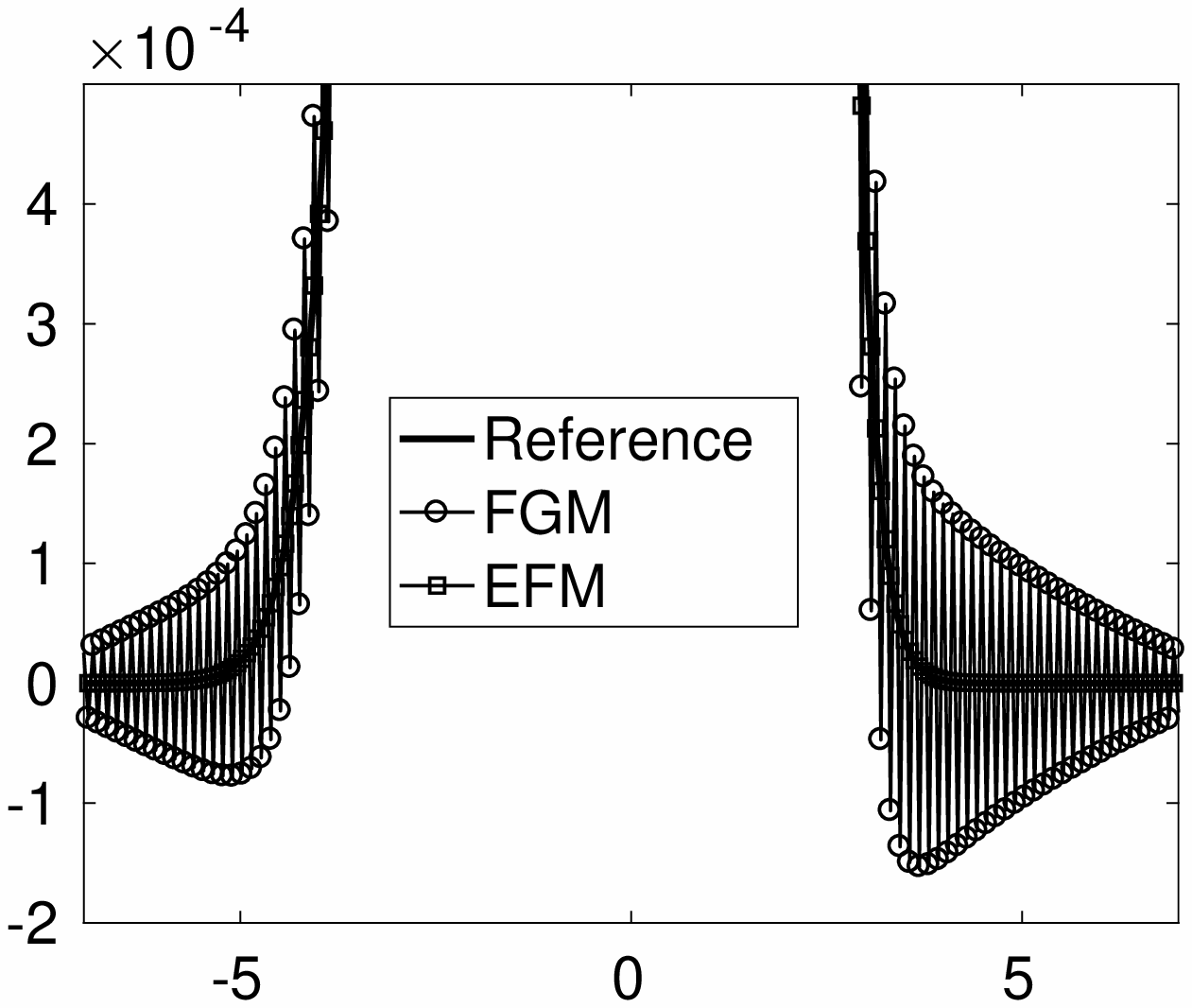}
    \caption{\label{fig:disc}Numerical solution of EFM and FGM for multiple $N$ values with the
      discontinuous initial value \eqref{eq:discontinuous} at time $t=0.5$ on different scale.}
\end{figure}

The profile of the reference solution is presented in Figure \ref{fig:Disc2D}, which is computed by
EFM with $N=2048$. Due to the discontinuity in the initial value, the spectral accuracy of FGM is
lost. In addition, the Gibbs phenomenon leads to oscillations in the initial value of FGM.  In
Figure \ref{fig:disc}, the plots around the discontinuity demonstrate that EFM has a much better
agreement as compared to FGM. The oscillations in FGM solutions exhibit large errors and the
amplitude of the oscillation decreases slowly as $N$ increases. On the contrary, there is no
oscillation for EFM around the discontinuity and the solution is always kept non-negative.

\paragraph{Example 4 ($3D$ BKW solution).}
The solution of this example is the exact $3D$ BKW solution, given by
\begin{equation}
  \f(t,\bv)=\frac{1}{(2\pi S)^{3/2}}\exp\left( -\frac{|\bv|^2}{2S} \right)
  \left( \frac{5S-3}{2S}+\frac{1-S}{2S^2}|\bv|^2 \right),
\end{equation}
where $S=1-2\exp(-t/6)/5$. Similarly to the $2D$ case, we first check the accuracy of EFM. At time
$t=0.01$, the $\ell_1$, $\ell_2$ and $\ell_\infty$ errors and the convergence rates are listed in
the Table \ref{tab:accuracyBKW3D}. Similar to the $2D$ case, the convergence rate is of the second
order and the errors are rather small.

\begin{table}[!ht]
  \centering
  \begin{tabular}{|c|c|c|c|c|c|c|}
    \hline
    $N$ & $\ell_1$ error  & rate & $\ell_2$ error  & rate & $\ell_{\infty}$ error & rate \\ \hline
    16  & $4.08\times10^{-3}$ &     & $3.08\times10^{-3}$ &     & $3.56\times10^{-3}$ &     \\ \hline 
    32  & $1.42\times10^{-3}$ & 1.52& $1.12\times10^{-3}$ & 1.47& $1.26\times10^{-3}$ & 1.50\\ \hline
    64  & $4.07\times10^{-4}$ & 1.80& $3.29\times10^{-4}$ & 1.76& $3.72\times10^{-4}$ & 1.76\\ \hline
    128 & $1.08\times10^{-4}$ & 1.91& $8.85\times10^{-5}$ & 1.90& $1.00\times10^{-4}$ & 1.89\\ \hline
  \end{tabular}
  \caption{\label{tab:accuracyBKW3D}The $\ell_1$, $\ell_2$ and $\ell_\infty$ errors and convergence
    rates for the BKW solution at time $t=0.01$ with $R=6$. }
\end{table}

\begin{figure}[!ht]
    \centering
    \subfigure{%
    \includegraphics[width=0.45\textwidth,clip]{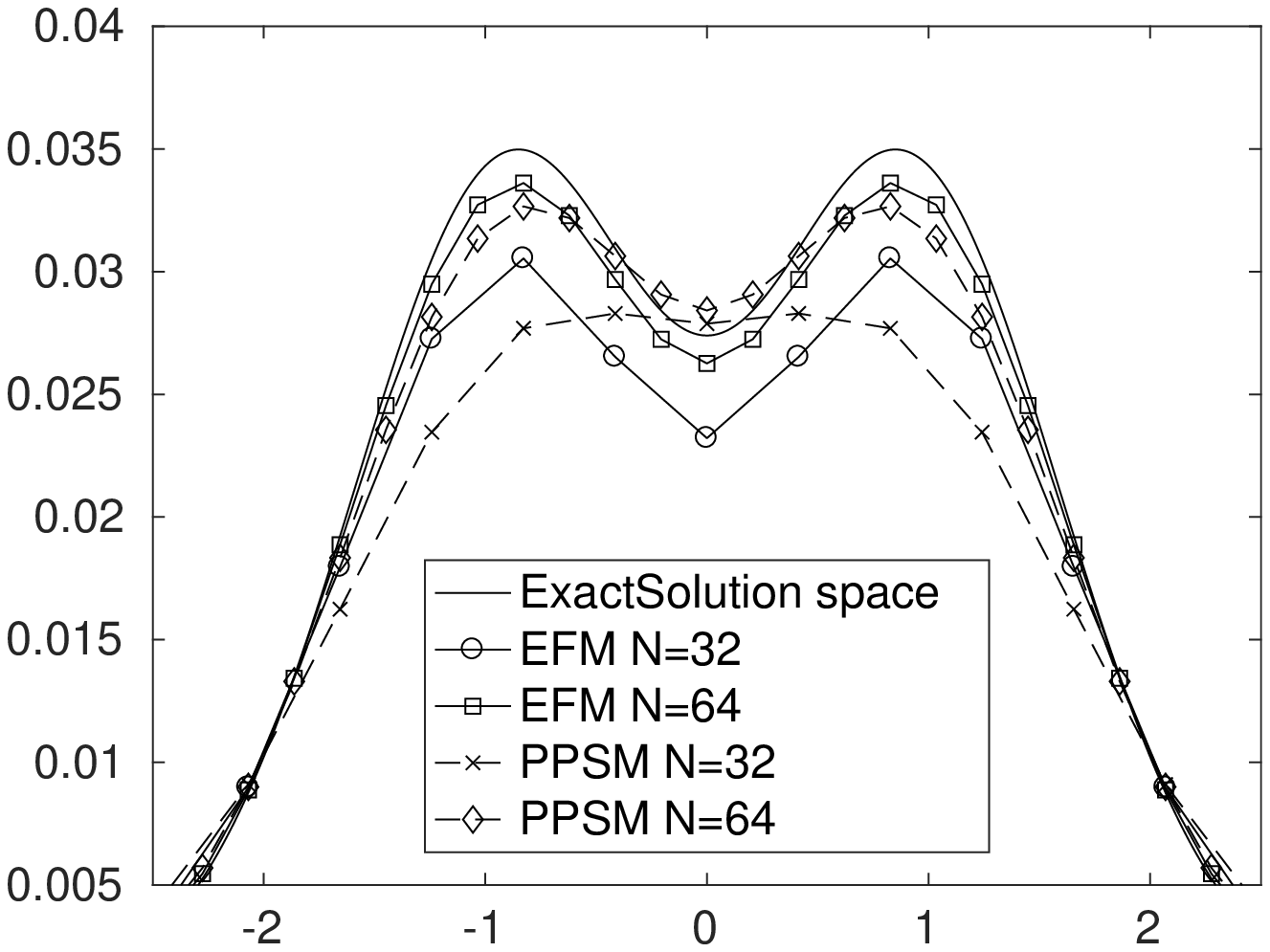}
    }\quad
    \subfigure{%
    \includegraphics[width=0.45\textwidth,clip]{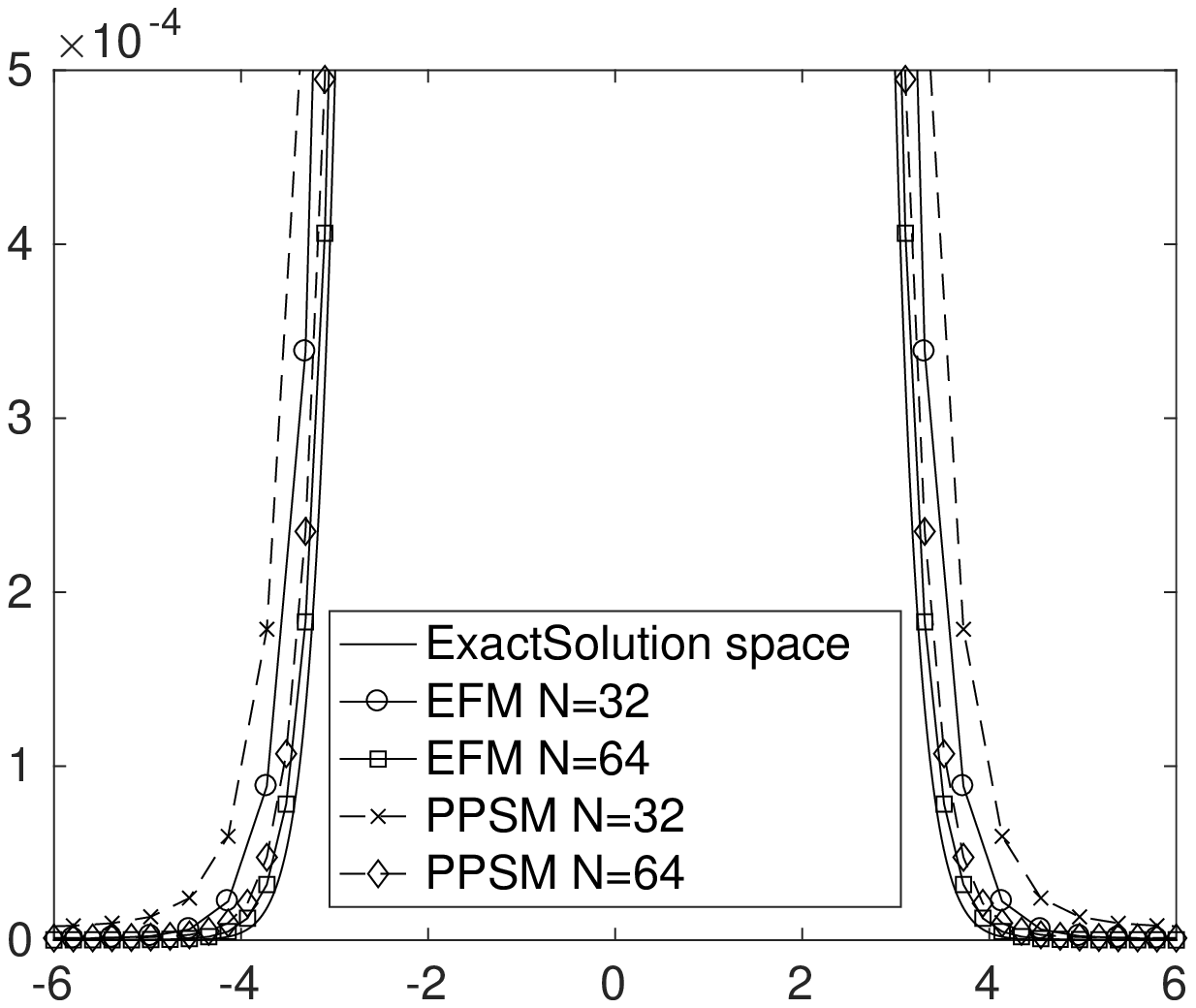}
    }
    \caption{\label{fig:PositiveVSHtheorem3D}Numerical solution of EFM and PPSM for multiple $N$
      values with the BKW solution at time $t=1$ on different scales.}
\end{figure}

As a comparison with PPSM, Figure \ref{fig:PositiveVSHtheorem3D}
presents the numerical solutions on the $v_1$ direction of PPSM and EFM at $t=1$ with $N=32$. The
plots clearly show that the smoothing filter used in EFM results in much less dissipation, thus
leading to a better agreement with the exact solution.

% \input{conclusion}
% vim: tw=70:spell
\section{Discussion}\label{sec:conclusion}

EFM proposed in this paper is a trade-off between accuracy and preservation of physical
properties. The resulting scheme can be viewed both as a discrete velocity method and a Fourier
method. In terms of the convergence rate, it is better than DVM but slower than FGM.  In terms of
physical properties, it guarantees positivity, mass conservation and a discrete H-theorem, while the
momentum and energy conservation is lost. Regarding the computational cost, fast algorithms in
\cite{Mouhot, Hu2016} remain valid for EFM. As to the future work, we plan to study how to
mitigate momentum and energy loss, where higher order accuracy is needed for long time simulation.
The numerical implementation of the spatially inhomogeneous setting is also in progress.

\section*{Acknowledgments}
The authors thank Jingwei Hu for discussion on the filter and fast algorithm on Boltzmann collision
term.

\bibliographystyle{plain}
\bibliography{article}
\end{document}